\documentclass[a4paper,french,twoside,11pt,centertags]{smfartVS}

\usepackage{smfenum,smfthm,amsmath,amssymb,mathrsfs,euscript,color}
\usepackage{stmaryrd,mathabx,upgreek}
\usepackage[latin1]{inputenc}
\usepackage[colorlinks]{hyperref}

\numberwithin{equation}{section} 
\theoremstyle{plain}

\newcounter{nonumber}

\newtheorem{theon}[nonumber]{Théorème}

\def\NN{\mathbf{N}}

\def\ZZ{\mathbf{Z}} 

\def\D{{\rm D}}

\def\F{{\rm F}}
\def\G{{\rm G}}
\def\H{{\rm H}}
\def\I{{\rm I}}
\def\J{{\rm J}}

\def\L{{\rm L}}
\def\M{{\rm M}}
\def\N{{\rm N}}
\def\O{{\rm O}}
\def\P{{\rm P}}
\def\Q{{\rm Q}}
\def\R{{\rm R}}
\def\SS{{\rm S}}

\def\U{{\rm U}}
\def\V{{\rm V}}
\def\W{{\rm W}}
\def\X{{\rm X}}

\def\Z{{\rm Z}}

\def\Dd{\mathscr{D}}
\def\Ee{\mathscr{E}}

\def\Hh{\mathscr{H}}

\def\Ss{\mathscr{S}}

\def\AA{\mathfrak{A}}

\def\a{\alpha} 
\def\b{\beta}

\def\g{\gamma}

\def\k{\kappa}
\def\l{\lambda}
\def\m{\mathfrak{m}}

\def\om{\omega}

\def\s{\sigma} 

\def\Om{\Omega}
\def\Si{\Sigma}

\def\ie{c'est-à-dire }

\def\rp{\rangle}
\def\>{\geqslant}
\def\<{\leqslant}
\def\odo{\otimes\dots\otimes}
\def\tdt{\times\dots\times}

\def\Hom{{\rm Hom}}
\def\End{{\rm End}}

\def\Mat{\mathscr{M}}
\def\GL{{\rm GL}}

\def\Ind{{\rm Ind}}
\def\ind{{\rm ind}}

\def\mult#1{{#1}^{\times}}

\def\blue#1{\textcolor{blue}{#1}}

\def\Dd{\EuScript{D}}
\def\Hh{\EuScript{H}}

\def\Ss{\EuScript{S}}

\def\AA{\EuScript{A}}
\def\CA{\EuScript{C}}

\def\XA{{\Irr}}
\def\IA{\EuScript{I}}
\def\JA{\EuScript{J}}
\def\MA{\EuScript{M}}

\def\RA{\EuScript{R}} 
\def\SA{\EuScript{S}}

\def\FB{\boldsymbol{k}}
\def\GB{\mathscr{G}}
\def\PB{\mathscr{P}}

\def\MB{\mathscr{M}}

\def\Div{\ZZ}
\def\Dive{\NN}

\def\ss{\mathfrak{s}}
\def\tt{\mathfrak{t}}

\def\CR{{\rm R}}

\def\AC{\textbf{\textsf{D}}}
\def\AZ{\textbf{\textsf{D}}}
\def\EC{\textbf{\textsf{E}}}

\def\Irr{{\rm Irr}}

\def\cusp{{\rm cusp}}
\def\scusp{{\rm scusp}}

\def\ip{\boldsymbol{i}}

\def\rp{\boldsymbol{r}}
\def\r{{\bf{r}}}

\def\st{{\rm st}}

\def\BJ{{\bf J}}

\def\bk{\boldsymbol{\k}}

\def\bt{\boldsymbol{\uptau}}
\def\ac{\textbf{\textsf{d}}} 
\def\KM{\textbf{\textsf{K}}}
\def\FM{\textbf{\textsf{F}}}
\def\fm{\textbf{\textsf{f}}}
\def\tm{\textbf{\textsf{t}}}

\def\inv{\bt}

\def\qr{q_{\s}}
\def\nr{\nu_{\s}}
\def\sy#1{\boldsymbol{[}#1\boldsymbol{]}}
\def\({\left(}
\def\){\right)}

\def\sc{{\rm sc}}

\def\flb{{\overline{\mathbf{F}}_\ell}}
\def\qlb{{\overline{\mathbf{Q}}_\ell}}
\def\zlb{{\overline{\mathbf{Z}}_\ell}}

\newcounter{intronum}
\def\theintronum{\arabic{intronum}}
\newenvironment{introsec}{\refstepcounter{intronum}
\noindent{\bf\theintronum.}}{\medskip}

\newcounter{notanum}
\def\thenotanum{\arabic{notanum}}
\newenvironment{notasec}{\refstepcounter{notanum}
\noindent{\bf\thenotanum.}}{\medskip}

 \author{Alberto M\'\i nguez}
 \address{Institut de Mathématiques de Jussieu, Université Paris 6, 
 4 place Jussieu, 75005, Paris, France. 
 URL: {\rm http://www.math.jussieu.fr/$\sim$minguez/}} 
 \email{minguez@math.jussieu.fr}

 \author{Vincent Sécherre}
 \address{Université de Versailles Saint-Quentin-en-Yvelines\\
 Laboratoire de Mathémati\-ques de Versailles\\
 45 avenue des Etats-Unis\\
 78035 Versailles cedex, France}
 \email{vincent.secherre@math.uvsq.fr}

\title{L'involution de Zelevinski modulo $\ell$}

\begin{altabstract}
Let $\F$ be a non-Archimedean locally compact field with residual characteristic 
$p$, let $\G$ be an inner form of $\GL_n(\F)$, $n\>1$ and let $\R$ be an 
algebraically closed field of characteristic diffe\-rent from $p$.
When $\R$ has characteristic $\ell>0$, the image of an irreducible smooth 
$\R$-representation $\pi$ of $\G$ by the Aubert involution need not be 
irreducible.
We prove that this image (in the Grothen\-dieck group of $\G$)
contains a uni\-que irreducible term $\pi^\star$ with the same 
cuspidal support as $\pi$. 
This defines an involution $\pi\mapsto\pi^\star$ on the set of isomorphism
classes of irreducible $\R$-representations of $\G$, that coincides with the 
Zelevinski involution when $\R$ is the field of complex numbers. 
The method we use also works for $\F$ a finite field of characteristic $p$,
in which case we get a similar result. 
\end{altabstract}

\long\def\MSC#1\EndMSC{\def\arg{#1}\ifx\arg\empty\relax\else
     {\par\narrower\noindent%
     2010 Mathematics Subject Classification: #1\par}\fi}

\long\def\KEY#1\EndKEY{\def\arg{#1}\ifx\arg\empty\relax\else
	{\par\narrower\noindent Keywords and Phrases: #1\par}\fi}

\begin{document}

\maketitle


\MSC 
22E50, 20G40
\EndMSC

\KEY 
Modular representations, $p$-adic reductive groups, 
finite reductive groups, Zelevinski involution, Alvis-Curtis duality, 
type theory 
\EndKEY


\section*{Introduction}

\begin{introsec}
Soit $\F$ un corps localement compact non archimédien, 
de caractéristique résiduelle $p$. 
Zelevinski \cite{Ze1} a défini une involution sur le groupe de 
Grothendieck des représentations complexes de longueur finie 
de $\GL_n(\F)$, pour $n\>1$, 
et a conjecturé que cette involution préserve l'irréduc\-ti\-bi\-li\-té. 
Moeglin et Waldspurger \cite{MW} ont prouvé cette conjecture 
pour les représentations irréducti\-bles de $\GL_n(\F)$ possédant 
un vecteur non nul invariant par un sous-groupe d'Iwahori. 
Précisé\-ment, ils ont montré que, 
si l'on applique le foncteur des invariants par un sous-groupe d'Iwahori, 
l'involution de Zelevinski se transforme, pour les modules sur 
l'algèbre de Hecke-Iwahori, 
en la torsion par une involution de cette algèbre. 
\end{introsec}

\begin{introsec}
En s'inspirant de la dualité d'Alvis-Curtis \cite{Alvis1,Alvis2,Curtis}, 
S.-I.~Kato
\cite{Kato} a défini une involution sur le groupe de 
Grothendieck des représentations complexes de longueur finie 
(et engendrées par leurs vecteurs invariants sous un sous-groupe d'Iwahori) 
d'un groupe réductif $p$-adique déployé.
Comme dans \cite{MW}, il utilise les propriétés du foncteur 
des invariants par un sous-groupe d'Iwahori \cite{Borel} pour 
prouver que cette involution préserve l'irréductibilité à un signe près.
\end{introsec}

\begin{introsec}
Aubert \cite{Aubert} a montré que la définition de Kato permet 
d'obtenir une involution sur le groupe de 
Grothendieck des représentations complexes de longueur finie 
de n'importe quel groupe réductif $p$-adique, et a prouvé que 
cette involution préserve l'irréductibilité à un signe près.
Dans le cas du groupe $\GL_n(\F)$, elle coïncide avec
l'involution de Zelevinski à un signe près, 
ce qui prouve la con\-jecture d'irréductibilité de 
Zelevinski pour toutes les représentations irréductibles complexes de 
$\GL_n(\F)$.
Procter \cite{Procter} avait déjà prouvé cette conjecture peu de temps 
auparavant 
au moyen de la théorie des types de Bushnell-Kutzko \cite{BK}, 
ce qu'on peut voir comme une géné\-ralisation de l'approche de 
Moeglin et Waldspurger à n'importe quel bloc de Bernstein de $\GL_n(\F)$.
\end{introsec}

\begin{introsec}
Au moyen de la théorie des systèmes de coefficients sur l'immeuble de 
Bruhat-Tits, 
Schnei\-der et Stuh\-ler \cite{SSihes} ont défini eux aussi une dualité
pour les représentations comple\-xes 
de longueur fi\-nie d'un groupe réduc\-tif $p$-adique. 
Ils prouvent qu'elle préserve l'irré\-ductibilité et qu'elle coïncide à un
signe près, au ni\-veau des groupes 
de Grothendieck, avec l'involution d'Au\-bert. 
On trouve une autre approche dans Bezrukavnikov \cite{Bez}, 
esquissée dans Bernstein \cite[IV.5.1]{Bnotes}.
\end{introsec}

\begin{introsec}
Vignéras \cite{VigZ} étend la question aux représentations des groupes 
réductifs $p$-adiques à coeffi\-cients dans un corps 
de caractéristique $\ell$ différente de $p$.
Dans ce contexte, la définition d'Aubert a toujours un sens et, 
pourvu que le groupe ait des 
sous-groupes discrets cocompacts (auquel cas la conjec\-tu\-re d'irréductibilité
générique est prouvée \cite{Datf}), elle définit toujours une involution \cite{MSb}. 
En revanche, quand $\ell$ est un nom\-bre premier non banal, 
il est facile de voir que cette involution ne préserve pas l'irréductibilité, 
même à un signe près. 
Quant à la défini\-tion de Zelevinski pour $\GL_n(\F)$, elle a toujours un sens 
elle aussi \cite{MSb} mais 
ne définit pas un automorphisme involutif ni ne 
préserve l'irré\-duc\-tibilité dans le cas non banal 
(remarque \ref{ZeleNonInv}). 
Dans le cas banal par contre, voir \cite{MSb} où l'on traite le cas de
$\GL_n(\F)$ et de ses formes intérieures.
\end{introsec}

\begin{introsec}
Reprenant l'approche de Schnei\-der-Stuh\-ler dans le contexte modulaire, 
Vignéras \cite{VigZ} prouve -- au moins lorsque le groupe a des 
sous-groupes discrets cocompacts -- que l'involution d'Aubert et celle de 
Schneider-Stuhler coïncident à un signe près au niveau des groupes 
de Grothendieck~et que, si $\pi$ est une repré\-sen\-tation 
irréductible d'un groupe réductif $p$-adique, son image par cette involution 
possède un seul terme irréductible de même support cuspidal que $\pi$. 
En outre, elle mon\-tre que ce terme irréductible est 
caractérisé comme étant l'unique quotient irréductible 
d'un certain espace de cohomologie associé à $\pi$
(voir \cite[Theorem 4.6]{VigZ}).
\end{introsec}

\begin{introsec}
\label{intro7}
Dans cet article, nous donnons une autre preuve du résultat de Vignéras pour 
$\GL_n(\F)$, $n\>1$ et ses for\-mes intérieures, suivant une 
approche fondée sur la théorie des types de Bushnell-Kutzko
dans l'esprit de \cite{Procter}. 
Ce travail se situe dans la continuité de nos travaux sur les 
représentations modulaires des formes intérieures de $\GL_n(\F)$
dans lesquels la théorie des types \cite{MSt} jour l'un des rôles principaux. 
Il forme avec \cite{MSb,MSc,MSt} 
un ensemble homogène décrivant de façon cohérente la théorie des 
représentations modulo $\ell\neq p$ des formes intérieures de $\GL_n(\F)$.
Dans le cas des représentations complexes, 
les liens unissant la théorie des systèmes de 
coefficients et la théorie des types ont été explorés dans 
\cite{BrAc,BrSch}. 
Nous décrivons notre stratégie ci-dessous.
\end{introsec}

\begin{introsec}
Soit $\G$ une forme intérieure de $\GL_n(\F)$, 
\ie un groupe de la forme $\GL_{m}(\D)$ où $m$ est un diviseur de $n$ et 
$\D$ une $\F$-algèbre à division centrale de degré réduit $d$ tels que 
$md=n$, soit $\R$ un corps algébriquement clos 
de caractéristique différente de $p$ 
et soit $\AC$ l'invo\-lution d'Aubert
(voir le paragraphe \ref{DefASS})
sur le groupe de Grothendieck des repré\-sen\-ta\-tions de lon\-gueur fi\-nie 
de $\G$ à coef\-ficients dans $\R$.
Notre résultat prin\-ci\-pal est le suivant (théorème \ref{ASG}). 

\begin{theon}
Soit $\pi$ une $\R$-représentation irréductible de $\G$, 
et soit $r(\pi)$ le nombre de termes de son support cuspidal. 
Il y a une unique repré\-sentation irréducti\-ble $\pi^\star$ de $\G$,
de même support cuspidal que $\pi$, telle que la quantité~: 
\begin{equation*}
\AC(\pi)- (-1)^{r(\pi)} \cdot \pi^\star 
\end{equation*}
dans le groupe de Grothendieck de $\G$
ne contienne pas de terme irréductible de même support cuspidal que $\pi$.
\end{theon}
\end{introsec}

\begin{introsec}
La première étape de notre démonstration consiste à se ramener -- par des 
raisonnements sur les supports cuspidal et supercuspidal -- au cas d'une 
représentation irréductible dont le support cuspidal est de la forme 
$\s_1+\dots+\s_r$, où $\s_1,\dots,\s_r$ sont des représentations 
irréductibles supercuspidales inertiellement équivalentes à une même 
représentation irréductible supercuspidale $\s$ de $\GL_{k}(\D)$ avec $kr=m$. 
Dans cette situation, la théorie des types de Bushnell-Kutzko fournit 
un foncteur $\FM$ de la catégorie des $\R$-représentations de $\G$ vers celle 
des modu\-les à droite sur une algèbre de Hecke $\Hh(\s,r)$. 
Contrairement à ce qui se passe en caractéristique nulle 
(voir \cite{BK,VS3}) ce foncteur n'est en général pas exact quand $\R$ est de
caractéristique $\ell>0$. 
Il est représentable par une représentation de type fini $\Q$ de $\G$ 
qui n'est en général pas projective dans la catégorie des $\R$-représentations 
de $\G$, mais qui est quasi-projective, ce qui entraîne les propriétés suivantes 
(voir \cite{MSt} et le paragraphe \ref{RepMod})~:
\begin{enumerate}
\item
$\FM$ induit une bijection entre les classes de représentations
irréduc\-ti\-bles de $\G$ dont le support cuspidal est de la forme~:
\begin{equation}
\label{OMS}
\s_1+\dots+\s_r,
\quad
\text{$\s_1,\dots,\s_r $ supercuspidales et inertiellement équivalentes à $\s$}
\end{equation}
et les $\Hh(\s,r)$-modules à droite simples~;
\item
$\FM$ est exact sur la sous-catégorie pleine $\Ee(\s,r)$ dont les objets sont 
les sous-quotients de sommes directes arbitraires de copies de $\Q$.
\end{enumerate}
\end{introsec}

\begin{introsec}
\label{intro9} 
Soit $\RA_{\s,r}$ le groupe de Grothendieck des représentations de longueur 
finie de $\Ee(\s,r)$, \ie le groupe abélien libre engendré par les classes de 
représentations irréduc\-ti\-bles de $\G$ dont le support supercuspidal est de 
la forme \eqref{OMS}.
Le foncteur $\FM$ induit un morphisme de $\RA_{\s,r}$ vers le groupe abélien
libre $\MA_{\s,r}$ engendré par les $\Hh(\s,r)$-modules à droite simples.
Notant $\RA_{\s}$ la somme directe des $\RA_{\s,r}$, $r\>0$ et définissant 
$\MA_{\s}$ de façon analogue, on obtient un morphisme de groupes de 
$\RA_{\s}$ vers $\MA_{\s}$, que l'on note encore $\FM$.
Son noyau, noté $\JA_{\s}$, est engendré par les classes de représentations
irréduc\-ti\-bles dans $\RA_{\s}$ dont le support super\-cuspidal est différent 
du support cuspidal. 
Ainsi, le théorème principal du paragraphe \ref{intro7} ci-dessus peut être reformulé 
comme suit (voir le paragraphe \ref{ConcPart}). 

\begin{theon}
Soit $\pi$ une $\R$-représentation irréductible dont le support supercuspidal 
est de la forme \eqref{OMS}.
Alors $(-1)^r\cdot\FM(\AC(\pi))$ est un module simple dans $\MA_{\s}$.
\end{theon}

La représentation irréductible $\pi^\star$ correspondant
bijectivement par $\FM$ à ce module simple
est l'unique repré\-sentation irréducti\-ble de même support cuspidal que
$\pi$ telle que $\AC(\pi)- (-1)^{r} \cdot \pi^\star$
ne contienne pas de terme irréductible de même support cuspidal que $\pi$.
\end{introsec}

\begin{introsec}
Passant maintenant au quotient, $\FM$ induit un isomorphisme de groupes $\fm$ de 
$\AA_{\s}$, le quotient de $\RA_{\s}$ par $\JA_{\s}$, vers $\MA_{\s}$.
L'involution $\AC$, laissant stable $\RA_{\s}$ et $\JA_{\s}$ d'après le 
corollaire \ref{astablej}, induit une involution $\ac$ sur $\AA_{\s}$.
Pour prouver le théorème du paragraphe \ref{intro9},
il s'agit de calculer l'iso\-mor\-phis\-me $\fm\circ\ac$ et pour cela de 
déterminer le comportement du foncteur $\FM$ vis-à-vis de l'induction et de la 
restriction paraboliques~:
c'est ce que nous faisons dans la section \ref{SEC3} (voir le 
co\-rol\-laire \ref{stableEsn} et la remarque \ref{vaval}).
Nous y montrons aussi que, pour sous-groupe parabolique $\P$ de $\G$ et tout 
facteur de Levi $\M$ de $\P$, 
le foncteur $\ip_{\M,\P}\circ\rp_{\M,\P}$ composé de l'induction et de la 
restriction paraboliques de $\M$ à $\G$ le long de $\P$ laisse stable la 
sous-catégorie $\Ee(\s,r)$ sur laquelle le foncteur $\FM$ est exact.
L'involution $\AC$ étant définie comme une somme alternée de tels foncteurs
(dans le groupe de Grothendieck), 
ceci nous permet d'obtenir une formule pour 
$\fm\circ\ac$ du côté de $\MA_{\s}$, ne dépendant que d'un invariant 
$q_{\s}\in\R$ associé à la représentation supercuspidale $\s$ (voir \eqref{blo} et
le paragraphe \ref{FormuleRaviolisVapeur}).
Grâce au principe du changement de groupe, cette formule permet de 
réduire le problème au cas où $\s$ est le caractère trivial de $\GL_1(\F')$ 
pour une extension finie convenable $\F'$ de $\F$,
ce que nous faisons au paragraphe \ref{GreatExpectations}.
\end{introsec}

\begin{introsec}
\label{intro11} 
Il ne reste maintenant plus qu'à prouver le théorème du paragraphe 
\ref{intro9} dans le cas où $\s$ est le caractère trivial de $\GL_1(\F)$ 
pour un corps localement compact non archimédien $\F$ quelconque, 
noté simplement $1$. 
L'induction parabolique définit une multiplication sur $\AA_{\s}$,
ce qui en fait une $\ZZ$-algèbre commutative (\S\ref{GeEm}).
On a une structure analogue de $\ZZ$-algèbre commutative sur $\MA_{\s}$, 
et les résultats de la section \ref{SEC3} montrent que $\fm$ 
est alors un isomorphisme d'algèbres. 
D'autre part l'algèbre de Hecke $\Hh(\s,r)$ est une algèbre de 
Hecke affine de type \textsf{A} naturellement munie d'un automorphisme 
involutif d'algèbre $\inv$ (voir le paragraphe \ref{hecke}).
La torsion des modules par cette involution définit une involution sur
$\MA_\s$, encore
notée $\inv$.
Leclerc, Thi\-bon et Vasse\-rot \cite{LTV} ont établi un algorithme permettant 
de calculer l'image par $\inv$ d'un module simple 
en déterminant le multi\-seg\-ment
apériodique lui correspondant.
Nous prouvons le résultat suivant, qui implique et précise le théorème du 
paragraphe \ref{intro9}. 

\begin{theon}
Soit $\pi$ une $\R$-représentation irréductible dont le support supercuspidal 
est de la forme \eqref{OMS}.
Alors $(-1)^r\cdot\FM(\AC(\pi))$ est égal au module simple $\inv(\FM(\pi))$.
\end{theon}
\end{introsec}

\begin{introsec}
Pour tout $\Hh(\s,r)$-module simple $\m$, posons 
$\tm(\m)=(-1)^r\cdot\inv(\m)$. 
Prolongeant par linéarité, on obtient un automorphisme involutif d'algèbre sur 
$\MA_\s$. 
Pour prouver que les isomorphismes d'algèbres $\fm\circ\ac$ et 
$\tm\circ\fm$ sont égaux, il suffit de prouver qu'ils coïncident sur un 
système générateur de $\AA_\s$.
Un système générateur bien adapté au problème est fourni par la base standard 
(voir le théorème \ref{basestd} ou \cite[Lemme 9.41]{MSc}).
Grâce à la propriété de multiplicativité de la base standard, 
il suffit de comparer $\fm\circ\ac$ et 
$\tm\circ\fm$ sur les représentations $\Z(\s,r)$ associées à des segments
(voir le paragraphe \ref{segments}).
\end{introsec}

\begin{introsec}
Le calcul de $\fm\circ\ac(\Z(\s,r))$ se fait d'abord quand $\s$ est le 
caractère trivial de $\mult\F$ 
-- auquel cas $\Z(1,r)$ n'est autre que le caractère trivial de $\GL_r(\F)$ -- 
grâce à un argument de relèvement à la 
caractéristique nulle (voir les paragraphes \ref{FormuleDZ} et \ref{CasUnip}). 
Ceci implique le théorème du paragraphe \ref{intro11} pour $\s$ trivial, 
lui-même impliquant (grâce à la méthode de changement de groupe) 
ce même théorème pour $\s$ quelconque. 
Ceci enfin implique en retour le théorème du paragraphe \ref{intro11} pour $\s$ 
quelconque (voir le paragraphe \ref{CasG}). 
\end{introsec}

\begin{introsec}
Terminons cette introduction par trois remarques. 
D'abord, notre méthode fonctionne aussi bien pour un corps $\F$ 
localement compact non archimédien de caractéristique résiduelle $p$ 
que pour un corps fini de carac\-té\-ris\-tique $p$. 
De fait, l'article est rédigé de façon uniforme en $\F$, 
qu'il soit fini ou $p$-adique.
Le seul passage où le cas fini nécessite d'être traité avant le cas $p$-adique 
est le début de la section \ref{SEC3} où nous étudions le comportement du
foncteur $\FM$ vis-à-vis de l'induction et de la restriction paraboliques. 
Dans le cas où $\F$ est fini, nos résultats généralisent un résultat de 
Ackermann et Schroll \cite{AS} qui traitent le cas unipotent, \ie le cas 
où $\s$ est trivial. 
\end{introsec}

\begin{introsec}
La seconde remarque concerne la section \ref{SEC3} dont l'intérêt, 
du point de vue de la théorie des types, dépasse le cadre de cet article. 
Le calcul des coinvariants de $\Q$ relativement au radical 
uni\-potent d'un sous-groupe parabolique effectué au paragraphe \ref{sec:Types} 
éclaire cer\-taines ques\-tions sou\-le\-vées ou partiellement résolues dans
la section $4$ de \cite{MSt}.
Notamment, le corollaire \ref{stableEsn} étend le domaine de validité de 
\cite[Corollaire 4.31]{MSt}
et le corollaire \ref{Tet} 
(qu'on pourrait raffiner en dé\-fi\-nissant une 
structure de bigèbre au moyen de la restriction parabolique) 
montre qu'un certain nombre de résultats 
impliquant induction et restriction paraboliques peuvent être transportés -- 
via la méthode de changement de groupe -- d'un $\s$ à un autre de même 
invariant $q_\s$.
Ceci sera utile dans des travaux ultérieurs. 
\end{introsec}

\begin{introsec}
Notre troisième et dernière remarque porte sur le calcul des multiplicités des 
représentations ir\-ré\-ductibles dans les représentations standards 
\eqref{defRSTD}.
Si $\pi$ est une $\R$-représentation irréductible dont le support 
cuspidal est de la forme \eqref{OMS}, nous montrons que sa multiplicité 
dans une re\-pré\-sentation standard ne dépend de $\s$ que par le biais de 
$q_\s$. 
Il s'ensuit grâce à \cite{CG,AM} que cette multiplicité ne dépend de $\s$ que 
par le biais de l'entier $e(\s)$ défini par \eqref{DEFe}.
Le fait était bien connu -- au moins dans le cas complexe -- mais n'était 
à notre connaissance écrit nulle part dans la littérature. 
Nous avons profité de l'appareil technique mis en place dans cet article pour 
le fai\-re. 
Nous remercions P.~Boyer, E.~Lapid, B.~Leclerc et M.~Tadi\'c de nous y avoir 
incité. 
\end{introsec}

\section*{Notations et conventions}

\begin{notasec}
\label{NotCorps}
On note 
$\NN$ l'ensemble des entiers naturels et
$\ZZ$ l'anneau des entiers relatifs.
\end{notasec}

\begin{notasec}
\label{in:partition}
Une \textit{composition} d'un entier $n\>1$ est une famille finie d'entiers 
$>0$ de som\-me $n$. 
\end{notasec}

\begin{notasec}
Pour un ensemble $\X$, on note $\ZZ(\X)$ le groupe 
abélien libre de base $\X$ constitué des applications de $\X$ dans $\ZZ$ 
à support fini et $\NN(\X)$ le sous-ensemble de $\ZZ(\X)$ 
constitué des applications à valeurs dans $\NN$.
Si $f,g\in\ZZ(\X)$, on note $f\<g$ si $g-f\in\NN(\X)$, 
ce qui définit une relation d'ordre partiel sur $\ZZ(\X)$.
\end{notasec}

\begin{notasec}
Dans tout cet article, $p$ est un nombre premier et 
$\R$ un corps algébriquement clos de caracté\-ristique dif\-fé\-rente de 
$p$.  
\end{notasec}

\begin{notasec}
Une $\R$-{\it re\-pré\-sen\-ta\-tion lisse} d'un groupe localement profini 
$\G$ est un mor\-phisme de groupes de 
$\G$ dans $\GL(\V)$, où $\V$ est un espace vectoriel sur $\R$, 
tel que tout vecteur de $\V$ ait un stabilisateur ouvert dans $\G$.
{\rm Dans cet article, toutes les représentations sont des 
  $\R$-re\-pré\-sen\-ta\-tions lisses.} 
Un $\R$-{\it caractère} de $\G$ est un mor\-phis\-me de $\G$ vers $\mult\R$ 
de noyau ouvert. 
Si aucune confusion n'est à craindre, on écrira \textit{caractère} et 
\textit{re\-présentation} plutôt que $\R$-carac\-tè\-re et 
$\R$-re\-pré\-sen\-ta\-tion. 
\end{notasec}

\begin{notasec}
Si $\pi$ est une représentation et si $\chi$ est un caractère de $\G$, 
on note $\pi\chi$ la représentation tordue définie par 
$g\mapsto\chi(g)\pi(g)$.
\end{notasec}


\begin{notasec}
On note $\Irr(\G,\R)$
l'ensemble des classes d'iso\-mor\-phisme des
re\-pré\-sen\-ta\-tions irréductibles de $\G$ et $\RA(\G,\R)$ le groupe 
de Gro\-then\-dieck de ses re\-pré\-sen\-ta\-tions de longueur finie, qui 
s'identifie au groupe abélien libre $\Div(\XA(\G,\R))$.
Le plus souvent, on omet\-tra $\R$ dans les notations. 
\end{notasec}

\begin{notasec}
Si $\pi$ est une représentation de longueur finie de $\G$, on désigne par 
$\sy{\pi}$ son image dans $\RA(\G)$.
En particulier, si $\pi$ est irréductible, $\sy{\pi}$ désigne 
sa classe d'isomorphisme. 
Lorsqu'aucune confusion ne sera possible, il nous arrivera d'identifier 
une représentation avec sa classe d'isomorphisme.
\end{notasec}

\begin{notasec}
\label{PF}
Dans cet article, $\F$ désigne~: 
\begin{itemize}
\item
ou bien un corps fini de caractéristique $p$, 
de cardinal noté $q=p^r$, $r\>1$, (et on dira qu'on est dans le \textit{cas fini});
\item
ou bien un corps localement compact non archi\-médien de corps 
résiduel de cardinal $q=p^r$, $r\>1$ (et on dira qu'on est dans le 
\textit{cas $p$-adique}). 
\end{itemize}
\end{notasec}

\begin{notasec}
On fixe une $\F$-algèbre à division centrale de dimension finie $\D$, 
de degré réduit noté $d$. 
Dans le cas fini, on a $d=1$ et $\D$ est égale à $\F$.
Pour tout entier $m\>1$, on note $\Mat_{m}(\D)$ la 
$\F$-algèbre des matrices carrées de taille $m$ à coefficients dans $\D$, 
et on note $\G_{m}$ le groupe $\GL_{m}(\D)$ de ses élé\-ments 
in\-ver\-si\-bles. 
Il est commode de convenir que $\G_{0}$ est le groupe trivial. 
La topologie sur $\F$ 
induit sur $\G_m$ une topologie en faisant un groupe localement profini. 
(Dans le cas fini, c'est la topologie discrète.)
\end{notasec}

\begin{notasec}
On note 
$|\ |_{\F}$ la valeur absolue nor\-malisée sur $\F$.
Dans le cas $p$-adique, c'est la valeur absolue
don\-nant à une uniformisante de $\F$ la valeur $q^{-1}$.
Dans le cas fini, c'est la valeur absolue triviale. 
Comme l'image de $q$ dans $\R$ est in\-ver\-si\-ble, 
elle définit un $\R$-caractère de $\mult\F$ noté $|\ |_{\F,\R}$.
Si l'on note $\N_{m}$ la norme réduite de $\Mat_{m}(\D)$ sur $\F$, 
l'ap\-pli\-ca\-tion $g\mapsto|\N_{m}(g)|_{\F,\R}$ est un 
$\R$-ca\-rac\-tè\-re de $\G_{m}$, qu'on notera simplement 
$\nu$. 
Dans le cas fini, $\nu$ est donc le ca\-rac\-tè\-re trivial de $\G_{m}$. 
\end{notasec}

\begin{notasec}
On note $\Irr$ la réunion des $\Irr(\G_{m})$ 
et $\RA$ la somme directe des $\RA(\G_m)$, pour $m\>0$.
Celle-ci s'identifie au groupe abélien $\Div(\Irr)$. 
Pour une représentation de lon\-gueur finie $\pi$ de $\G_m$, 
on pose $\deg(\pi)=m$, qu'on ap\-pel\-le le {\it degré} de $\pi$. 
L'application $\deg$ fait de $\RA$ un $\ZZ$-module gradué.
\end{notasec}

\section{Préliminaires} 

Pour plus de détails sur les résultats de cette section nous renvoyons le 
lecteur à \cite{MSc} dans le cas $p$-adique et à \cite{MSf} dans le cas fini. 

\subsection{Induction et restriction paraboliques}
\label{GeEm}

Si $\a=(m_{1},\ldots,m_{r})$ est une composition de $m$, 
il lui correspond le 
sous-groupe de Levi standard $\M_{\a}$ de $\G_{m}$ constitué des matrices 
diagonales par blocs de tailles $m_{1},\ldots,m_{r}$ respectivement, que 
l'on identifie naturellement au produit 
$\G_{m_{1}}\times\cdots\times\G_{m_{r}}$. 
On note $\P_{\a}$ 
le sous-groupe para\-bo\-li\-que de $\G_{m}$ de facteur de Levi 
$\M_{\a}$ formé des matrices tri\-an\-gu\-lai\-res supérieures 
par blocs de tailles $m_{1},\ldots,m_{r}$ respectivement, et on note 
$\U_{\a}$ 
son radical unipotent.

On fixe une racine carrée de $q$ dans $\R$.
On note $\ip_\a$ le foncteur d'in\-duc\-tion para\-bo\-li\-que
(norma\-li\-sé, dans le cas $p$-adique, 
relativement au choix de cette racine)
de $\M_{\a}$ à $\G_{m}$ le long de $\P_{\a}$,
et on note 
$\rp_\a$ son adjoint à gau\-che, \ie le foncteur de restriction 
para\-bo\-li\-que lui correspondant.
Ces foncteurs sont exacts,
et préservent l'admissibilité et le fait d'être de longueur finie.

Si, pour chaque $i\in\{1,\ldots,r\}$, on a une 
représentation $\pi_{i}$ de $\G_{m_i}$, on note~: 
\begin{equation}
\label{VentreDieu}
\pi_1\times\cdots\times\pi_r=\ip_{\a}(\pi_1\otimes\cdots\otimes\pi_r).
\end{equation}
Si les $\pi_{i}$ sont de longueur finie, 
la représentation semi-simplifiée 
$\sy{\pi_1\times\cdots\times\pi_r}$
ne dépend que de $\sy{\pi_{1}},\dots,\sy{\pi_{r}}$.
L'ap\-pli\-cation~:
\begin{equation*}
\label{VentreDieuGris}
(\sy{\pi_1},\dots,\sy{\pi_r})\mapsto\sy{\pi_1\times\cdots\times\pi_r}
\end{equation*}
induit par linéarité une application linéaire 
de $\RA(\G_{m_1})\times\dots\times\RA(\G_{m_r})$ dans 
$\RA(\G_{m})$,
faisant de $\RA$ une $\ZZ$-al\-gè\-bre commutative 
(voir \cite[Proposition 2.6]{MSc} dans le cas $p$-adique)
gra\-duée. 

\subsection{Représentations cuspidales et supercuspidales}
\label{DefRepCusp}

Une représentation irréduc\-tible de $\G_m$, $m\>1$, 
est dite {\it cus\-pi\-da\-le} 
si elle n'apparaît comme sous-représentation 
d'aucune induite de la forme \eqref{VentreDieu} avec $r\>2$, et 
elle est dite \textit{supercuspidale} si elle n'apparaît comme sous-quotient 
d'aucune induite de la forme \eqref{VentreDieu} avec $\pi_1,\dots,\pi_r$ 
irréductibles et $r\>2$.  
(Il n'est pas nécessaire de supposer que $\pi_1,\dots,\pi_r$ 
sont irréductibles dans la définition précédente~; 
voir \cite[Proposition 11.1]{SSb}.)

On note ${\CA}$ le sous-ensemble de $\XA$ formé 
des clas\-ses de re\-pré\-sen\-ta\-tions irréductibles cuspidales, 
et $\SA$ le sous-ensemble de $\CA$ formé des classes de représentations 
supercuspidales.

Pour le résultat suivant, on renvoie à 
\cite{MSc} théorèmes 2.1 et 8.16 dans le cas $p$-adique, 
et à \cite{MSf} théorèmes 2.2 et 2.5 dans le cas fini.

\begin{prop}
Soit une représentation irréductible $\pi\in\Irr$. 
\begin{enumerate}
\item 
Il existe une unique somme
$\cusp(\pi)=\s_1+\dots+\s_r\in\NN(\CA)$,
appelée support cuspidal de $\pi$, 
telle que $\pi$ soit isomorphe à un quotient de $\s_1\tdt\s_r$.
\item 
Il existe une unique somme
$\scusp(\pi)=\om_1+\dots+\om_n\in\NN(\SA)$,
appelée support super\-cuspidal de $\pi$, 
telle que $\pi$ soit isomorphe à un sous-quo\-tient de
$\om_1\tdt\om_n$.
\end{enumerate}
\end{prop}

Soit $\s$ une représentation irréductible cuspidale, de degré $m\>1$.
Dans le cas $p$-adique, il lui correspond (via la théorie des types) deux 
entiers $f(\s),s(\s)\>1$ (voir \cite[3.4]{MSt}).
On pose~:
\begin{eqnarray}
\nr^{} & = & 
\left\{
\begin{array}{ll}
\nu^{s(\s)} & \text{dans le cas $p$-adique,} \\
\text{le caractère trivial} & \text{dans le cas fini,}
\end{array}
\right. \\
\label{blo}
\qr^{} & = &
\left\{
\begin{array}{ll}
q^{f(\s)} & \text{dans le cas $p$-adique,} \\
q^{\deg(\s)} & \text{dans le cas fini.}
\end{array}
\right. 
\end{eqnarray}
Pour harmoniser les notations, 
on pose $f(\s)=\deg(\s)$ dans le cas fini, 
de sorte qu'on a $\qr^{}=q^{f(\s)}$ dans tous les cas. 
Enfin on pose~:
\begin{equation}
\label{DEFe}
e(\s) = 
\left\{
\begin{array}{l}
\text{$0$ si $\R$ est de caractéristique nulle,} \\
\text{le plus petit 
$k\>2$ tel que 
$1+\qr^{}+\dots+\qr^{k-1}=0$ dans $\R$
sinon.}
\end{array}
\right. 
\end{equation}

Pour tout entier $n\>2$, l'induite~:
\begin{equation*}
\s \times \s \nr^{} \times \dots \times \s \nr^{n-1}
\end{equation*}
contient un sous-quotient irréductible cuspidal si et seulement si 
$\R$ est de caractéristique $\ell>0$ et s'il existe $r\>0$ tel que 
$n=e(\s)\ell^r$ 
(voir \cite[Proposition 6.4]{MSc} et \cite[Paragraphe 1.4]{MSf}). 
Dans ce cas, le sous-quotient cuspidal est unique, 
et il apparaît avec multiplicité $1$ dans l'induite. 
On le note~:
\begin{equation*}
\st_r(\s).
\end{equation*}
Le théorème ci-dessous donne une classification des représen\-ta\-tions 
irréductibles cuspidales en fonc\-tion des supercuspidales 
(voir \cite[Théorème 6.14]{MSc} et \cite[Théorème 1.4]{MSf}).

\begin{theo}
\label{AppCuspSuper}
\begin{enumerate}
\item 
L'application~:
\begin{equation*}
(\s,r)\mapsto\st_r(\s) 
\end{equation*}
est une surjection de $\Ss\times\ZZ_{\>0}$ sur $\CA\smallsetminus\SA$. 
\item 
Pour que
deux couples $(\s,r),(\s',r')\in\Ss\times\ZZ_{\>0}$ aient la 
même image par cette application, il faut et il suffit que 
$r'=r$ et qu'il existe un entier $i\in\ZZ$ tel que $\s'$ soit isomorphe à 
$\s\nr^i$.
\end{enumerate}
\end{theo}

Etant donnée une représentation irréductible cuspidale $\s$, on pose~: 
\begin{equation}
\label{classeinert}
\Om_\s = 
\left\{
\begin{array}{ll}
\{[\s]\} & 
\text{dans le cas fini,} \\
\{[\s\chi]\ |\ \text{$\chi$ caractère non ramifié de $\G_{\deg(\s)}$} \} & 
\text{sinon.}
\end{array}
\right. 
\end{equation}
Dans le cas $p$-adique, deux représentations irréductibles cuspidales 
$\s,\s'$ telles que $\Om_{\s}=\Om_{\s'}$
sont dites \textit{inertiellement équivalentes}.
Par commodité, nous étendrons cette définition au cas fini. 

\begin{theo}[{\cite[Théorème 4.18]{MSt}, \cite[Proposition 3.3]{MSf}}]
\label{DEC} 
Soient 
$\s_1,\dots,\s_r$ des repré\-sen\-ta\-tions 
cuspidales deux à deux non 
inertiellement équivalentes. 
Pour chaque $i\in\{1,\dots,r\}$, on fixe un support cuspidal 
$\ss_i$ formé de représentations inertiellement équivalentes 
à $\s_i$. 
\begin{enumerate}
\item 
Pour chaque entier $i$, soit $\pi_i$ une représentation irréductible de 
support cuspidal $\ss_i$.
Alors l'induite $\pi_1\times\dots\times\pi_r$ est irréductible.
\item
Soit $\pi$ une représentation irréductible de 
support cuspidal $\ss_1+\dots+\ss_r$.
Alors il existe des représentations $\pi_1,\dots,\pi_r$, 
uniques à isomorphisme près, telles que $\pi_i$ soit de support 
cuspidal $\ss_i$ pour chaque $i$ 
et telles que $\pi_1\times\dots\times\pi_r$ soit 
isomorphe à $\pi$.
\end{enumerate}
\end{theo}

On a aussi un variante supercuspidale de ce théorème.

\begin{theo}[{\cite[Théorème 8.19]{MSc}, \cite[Proposition 1.8]{MSf}}]
\label{DEC1} 
On reprend les hypothèses du théorème \ref{DEC}, en supposant en outre 
que $\s_1,\dots,\s_r$ sont supercuspidales. 
\begin{enumerate}
\item 
Pour chaque entier $i$, soit $\pi_i$ une représentation irréductible de 
support supercuspidal $\ss_i$.
Alors l'induite $\pi_1\times\dots\times\pi_r$ est irréductible.
\item
Soit $\pi$ une représentation irréductible de 
support supercuspidal $\ss_1+\dots+\ss_r$.
Il existe des représentations $\pi_1,\dots,\pi_r$, 
uniques à isomorphisme près, telles que $\pi_i$ soit de support 
supercuspidal $\ss_i$ pour chaque $i$ 
et telles que $\pi_1\times\dots\times\pi_r$ soit 
isomorphe à $\pi$.
\end{enumerate}
\end{theo}

\subsection{L'algèbre de Hecke}
\label{hecke}

Soient $n\>1$ et $u\in\mult\CR$. 
On note $\Hh(n,u)$ la $\CR$-algèbre engendrée 
par les symboles $\SS_1,\dots,\SS_{n-1}$ avec les relations~:
\begin{alignat}{2}
\label{eq:relations}
(\SS_i+1)(\SS_i-u)&=0, &\quad\quad& i\in\{1,\dots, n-1\},\\
\label{R1aff}
\SS_i\SS_j&=\SS_j\SS_i, && |i-j|\>2,\\
\label{R3aff}
\SS_i\SS_{i+1}\SS_i&=\SS_{i+1}\SS_i\SS_{i+1}, && i\in\{1,\dots, n-2\}.
\end{alignat}
Il y a donc une involution $\inv$ de $\Hh(n,u)$ définie par~:
\begin{equation}
\label{invf}
\SS_i\mapsto -\SS_{n-i}+u-1
\quad
i\in\{1,\dots,n-1\}.
\end{equation}
Puis on note $\widetilde{\Hh}(n,u)$ la $\CR$-algèbre engendrée 
par les symboles $\SS_1,\dots,\SS_{n-1}$ 
et $\X_1,\dots,\X_n$ et leurs inverses 
avec les relations \eqref{eq:relations} à \eqref{R3aff} auxquelles s'ajoutent 
les relations~: 
\begin{alignat}{2}
\X_i\X_j&=\X_j\X_i, &\quad\quad& i,j\in\{1,\dots,n\},\\
\X_j\SS_i&=\SS_i\X_j, && i\notin\{j,j-1\},\\
\label{Rderaff}
\SS_i\X_{i}\SS_i&=u\X_{i+1}, && i\in\{1,\dots,n-1\}.
\end{alignat}
La première s'identifie à une sous-algèbre de la seconde~;
on note $\inv$ l'involution de 
$\widetilde{\Hh}(n,u)$ défi\-nie par \eqref{invf} et~:
\begin{equation}
\label{invp}
\X_j\mapsto\X_{n+1-j},
\quad
j\in\{1,\dots,n\}.
\end{equation}

Soit $\s$ une représentation irréductible cuspidale de degré $m\>1$, 
et soit $\G=\G_m$. 
Fixons un entier $n\>1$.

Dans le cas $p$-adique, d'après \cite[Théorème 3.11]{MSt}, 
il existe un sous-groupe ouvert compact $\J$ de $\G$ 
et une représentation irréductible $\l$ de $\J$ tels que les 
représentations irréductibles de $\G$ dont la restriction à $\J$
admette $\l$ comme sous-représentation sont exactement les 
$\s\chi$ pour $\chi$ décrivant les caractères non ramifiés de $\G$. 
Un tel couple $(\J,\l)$ est appelé un type simple (maximal) pour $\s$.
On note~:
\begin{equation*}
\Si=\ind^\G_\J(\l)
\end{equation*}
l'induite compacte de $\l$ à $\G$. 
Dans le cas fini, on pose simplement $\Si=\s$.

On note $\Hh(\s,n)$ 
l'algèbre des endomorphismes de l'induite 
$\Si^{\times n}=\Si\tdt\Si$, 
le produit de $n$ copies de $\Si$.
D'après \cite[Proposition 4.18]{MSt} dans le cas $p$-adique et \cite[§5]{DippFlei1} 
dans le cas fini, il y a un isomorphisme naturel~:
\begin{equation}
\label{Hodkann}
\Psi_{\s,n} : \Hh(\s,n) \to
\left\{
\begin{array}{ll}
\widetilde{\Hh}(n,\qr)  & \text{dans le cas $p$-adique}, \\
\Hh(n,\qr) & \text{dans le cas fini}.
\end{array}
\right.
\end{equation}
Dans les deux cas, on a défini une représentation $\Si$ de $\G$ et, 
pour tout entier $n\>1$, une algèbre $\Hh(\s,n)$, 
isomorphe à une algèbre de Hecke de paramètre $\qr$,
et munie d'une involution $\bt$.

Si $\m$ est un $\Hh(\s,n)$-module à droite, on note $\inv(\m)$ 
le $\R$-espace vectoriel $\m$ muni d'une structure de $\Hh(\s,n)$-module à
droite par~:
\begin{equation*}
(x,h) \mapsto x*{\bt}(h)
\end{equation*}
pour tous $x\in\m$ et $h\in\Hh(\s,n)$, 
où $*$ désigne l'action de $\Hh(\s,n)$ sur le module $\m$.

Notons $\MA_{\s}$ la somme directe, portant sur $n\>0$, 
des groupes de Grothendieck des caté\-go\-ries des $\Hh(\s,n)$-modules 
à droite de dimension finie.

Pour toute composition $\a=(n_1,\dots,n_r)$ de $n$, on note 
$\Hh(\s,\a)$ la sous-$\R$-algèbre de $\Hh(\s,n)$ en\-gen\-drée par les 
$\SS_i$ tels que $i\notin\{n_1+n_2+\dots+n_k\ |\ 1\<k\<r-1\}$ --- 
auxquels on ajoute tous les $\X_j$ dans le cas $p$-adique.
On note ${\bf r}_\a$ 
le foncteur de restriction de $\Hh(\s,n)$ à $\Hh(\s,\a)$ 
et ${\bf i}_\a$ son adjoint à droite.
Ces deux foncteurs sont exacts.

De façon analogue au paragraphe \ref{GeEm}, on munit $\MA_\s$ 
d'une structure de $\ZZ$-algèbre commutative graduée~:
si, pour chaque $i\in\{1,\ldots,r\}$, on a un
$\Hh(\s,n_i)$-module à droite de dimension finie $\m_i$, 
on pose~:
\begin{eqnarray*}
\label{VentreDieuHmod}
\m_1\times\cdots\times\m_r
& =& {\bf i}_{\a}(\m_1\otimes\cdots\otimes\m_r) \\
& =& \Hom_{\Hh(\s,\a)}(\Hh(\s,n),\m_1\otimes\cdots\otimes\m_r).
\end{eqnarray*}

La sous-algèbre $\Hh(\s,\a)$ n'est en général pas stable par 
${\bt}$~: son image est $\Hh(\s,\a')$, où $\a'$ est la composition 
$(n_r,\dots,n_1)$.
Par conséquent, 
on a un isomorphisme fonctoriel de $\Hh(\s,n)$-modules
entre $\inv(\m_1\times\cdots\times\m_r)$ et 
$\inv(\m_r)\times\cdots\times\inv(\m_1)$
qui, après semi-simplification, donne l'égalité~:
\begin{equation*}
\inv(\m_1\times\cdots\times\m_r) = 
\inv(\m_1)\times\cdots\times\inv(\m_r)
\end{equation*}
dans $\MA_\s$, faisant de $\inv$ une involution de $\ZZ$-algèbre de $\MA_\s$. 

\subsection{Représentations et modules}
\label{RepMod}

On reprend les notations du paragraphe précédent. 
Soit $\Ee(\s,n)$ la sous-catégorie pleine de la catégorie des 
représentations de $\G_{mn}$ dont les objets sont les sous-quotients de 
sommes arbitraires de copies de $\Si^{\times n}$. 
Pour le résultat suivant, on renvoie à 
\cite[\textsection 4.1]{MSt} et \cite[\textsection 3]{MSf}.

\begin{theo}
\label{qptf}
\begin{enumerate}
\item
La représentation $\Si^{\times n}$ est quasi-projective \cite{Vigs,MSt}. 
\item
Le foncteur~:
\begin{equation}
\label{BIJ}
\FM : \pi\mapsto\Hom_\G(\Si^{\times n},\pi)
\end{equation}
est un foncteur exact de 
$\Ee(\s,n)$ dans la catégorie des $\Hh(\s,n)$-modules à droite.
\item
Ce foncteur 
induit une bijection 
entre les classes de représentations irréduc\-ti\-bles de $\G_{mn}$ dont le 
support cuspidal appartient à $\Dive(\Om_\s)$ 
et les $\Hh(\s,n)$-modules à droite simples.
\end{enumerate}
\end{theo}

On pose~:
\begin{eqnarray*}
\XA_{\s,n}^{} & = & \{\pi\in\XA\ |\ \text{$\pi$ est un sous-quotient de 
$\s_1 \tdt \s_n$ avec $\s_1,\dots,\s_n\in\Om_\s$} \}, \\
\XA_{\s,n}^{*} & = & 
\{\pi\in\XA\ |\ \text{$\pi$ est un quotient de 
$\s_1 \tdt \s_n$ avec $\s_1,\dots,\s_n\in\Om_\s$} \}.
\end{eqnarray*}
On note $\XA_{\s}^{}$ la réunion des $\XA_{\s,n}^{}$ pour $n\>0$,
et on définit $\XA_{\s}^{*}$ de façon analogue.
On note $\RA_\s$ la sous-$\ZZ$-algèbre de $\RA$ en\-gen\-drée par $\XA_{\s}$, 
et on note $\JA_\s$ l'idéal de $\RA_\s$ engendré par les
$\pi\in\Irr_{\s}^{}$ tels que $\pi\notin\Irr_{\s}^{*}$ 
(\ie tels que $\cusp(\pi)\notin\Dive(\Om_{\s})$).



Les foncteurs définis pour tout $n\>0$ par \eqref{BIJ} 
induisent un morphisme surjectif de groupes~:
\begin{equation}
\label{Falbala}
\FM : \RA_{\s} \to \MA_{\s}
\end{equation}
(encore noté $\FM$ par abus de notation) de noyau $\JA_\s$.

\begin{rema}
On verra plus loin (voir proposition \ref{Tet}) que $\FM$ est un morphisme 
d'algèbres. 
Après avoir défini dans la section \ref{sec:AC} une involution $\AC$ sur 
$\RA_\s$, on verra dans la section \ref{four}  
que c'est même | \textit{à un signe près} | un morphisme 
d'algèbres à involution lorsqu'on munit $\MA_\s$ de $\bt$.
\end{rema}

\section{Dualité}\label{sec:AC}

Dans cette section, on introduit un automorphisme involutif d'algèbre de
$\RA$, noté $\AC$ et appelé involution d'Aubert. 
Il ne préserve pas l'irréductibilité (même au signe près)
mais on montre au théo\-rè\-me \ref{ASG} que, pour toute 
représentation irré\-ductible $\pi$, son image $\AC(\pi)$ est une 
combinaison li\-né\-aire de re\-pré\-sen\-tations irréductibles dont une 
seule, notée $\pi^\star$, a le même support cuspidal que $\pi$. 

\subsection{L'involution $\AC$ de $\RA$}
\label{DefASS}

On définit un endomorphisme de groupe de $\RA$ 
en associant à toute représentation ir\-ré\-ductible $\pi\in\XA$ 
la représentation virtuelle dans $\RA$~:
\begin{equation*}
\AC(\pi)=\sum\limits_{\a}(-1)^{r(\a)}\cdot\ip_\a\circ\rp_\a(\pi)
\end{equation*}
où $\a$ décrit l'ensemble des compositions 
de $\deg(\pi)$ et où $r(\a)$ est le nombre de termes de $\a$.

\begin{prop}
\label{propAC}
L'application $\AC$ est un automorphisme involutif de $\ZZ$-algèbre
de $\RA$. 
\end{prop}

\begin{proof}
La preuve donnée dans \cite[§8]{DM} dans le cas fini pour les représentations 
complexes est encore valable pour les représentations modulaires. 

Dans le cas $p$-adique, 
voir \cite[Théorème 1.7]{Aubert} et \cite[Proposition A.2]{MSb}.
\end{proof}

\begin{defi}
Pour toute représentation irréductible $\pi\in\XA$, 
on note $r(\pi)$ le nombre de termes du support cuspidal de $\pi$. 
\end{defi}

\begin{rema}
\label{remAC}
Dans le cas fini, cette dualité a été introduite simul\-tanément par Alvis 
et Cur\-tis (voir \cite{Alvis1,Curtis} et \cite[\S8]{DM})
pour les représentations complexes de groupes plus généraux.
Elle pré\-ser\-ve l'irréductibilité à un signe près, \ie que, pour toute 
représentation irréductible complexe $\pi\in\XA$, on a~:
\begin{equation*}
(-1)^{r(\pi)}\cdot\AC(\pi)\in\XA.
\end{equation*}
La dualité a été étendue par Ca\-ba\-nes et Rickard \cite{CR} 
à des représentations à coefficients dans un anneau commutatif 
quelconque dans lequel $p$ est inversible.

Dans le cas $p$-adique, $\AZ$ a été introduite par Aubert 
\cite{Aubert} pour les représentations complexes de groupes réductifs 
$p$-adiques
(voir aussi \cite{Ze1,Tadic,BaduRena} pour le cas du groupe $\GL_m(\D)$). 
Elle préserve aussi l'irréductibilité à un signe près.
Voir \cite[Appendice A]{MSb} dans le cas modulaire banal.
\end{rema}

Dans le cas modulaire, $\AC$ ne préserve plus l'irréductibilité, pas même à un 
signe près. 

\begin{exem}
On suppose que la caractéristique $\ell$ de $\R$ divise $q+1$.
On note $1$ le caractère tri\-vial de $\F^\times$ 
(on a donc $e(\s)=2$ quand $\s$ est le caractère $1$).
On note $1_2$ le ca\-rac\-tère trivial de $\GL_2(\F)$ et $\nu_2$ le caractère 
$g\mapsto|\det(g)|_{\F,\R}$. 
Soit $\pi$ la représentation de $\GL_2(\F)$ sur l'espace des fonc\-tions 
lo\-ca\-lement constantes sur la droite projective sur $\F$, à valeurs dans $\R$. 
On a~:
\begin{equation*}
[\pi] =
1_2+\st_0(1)+\nu_2
\end{equation*}
dans $\RA$, et $\st_0(1)$ est cuspidale \cite{VigGL2} (voir aussi le théorème 
\ref{AppCuspSuper}). 
On a ainsi~:
\begin{eqnarray*}
\AC(1_2)&=&\nu_2+\st_0(1),\\
\AC(\st_0(1))&=&-\st_0(1).
\end{eqnarray*}
\end{exem}

Le but de cet article est de montrer le théorème suivant. 

\begin{theo}
\label{ASG}
Soit $\pi\in\Irr$.
Alors il y a une unique repré\-sentation irréductible $\pi^\star$
de même support cuspidal que $\pi$ telle que~: 
\begin{equation*}
\AC(\pi)- (-1)^{r(\pi)} \cdot \pi^\star\in\RA
\end{equation*}
ne contienne pas de terme irréductible de même support cuspidal que $\pi$.
\end{theo}

\begin{rema}
Il est clair que tous les sous-quotients irréductibles de $\AC(\pi)$ ont même 
support supercuspidal que $\pi$. Le problème est de montrer qu'un seul parmi
eux a le même support cuspidal que $\pi$. 
\end{rema}

\subsection{Premières réductions du problème}

Soit $\s$ une représentation irréductible cuspidale. 
On rappelle que $\RA_\s$ est la sous-algèbre de 
$\RA$ engendrée par $\XA_{\s}$, et que $\JA_\s$ est 
l'idéal de $\RA_\s$ engendré par les
$\pi\in\Irr_{\s}^{}$ telles que $\pi\notin\Irr_{\s}^*$.
On note aussi $\IA_\s$ l'idéal de $\RA_\s$ engendré par 
l'ensemble $\XA_{\s}$ privé du caractère trivial de $\G_0$.

On renvoie au paragraphe \ref{DefRepCusp} pour la définition de la notation 
$\st_{r}(\s)$, $r\>0$. 
Par commodité, on pose aussi $\st_{-1}(\s)=\s$. 

\begin{lemm}
\label{lemmeclef}
Soit $\pi$ une représentation irré\-ductible telle que 
$\cusp(\pi)\in\Dive(\Om_{\s})$.
Si $\tau$ est un terme irréductible de $\AC(\pi)$, alors~:
\begin{equation*}
\cusp(\tau) \in 
\Dive(\Om_{\s}) + \Dive(\Om_{\st_0(\s)}) + \Dive(\Om_{\st_1(\s)}) + 
\Dive(\Om_{\st_2(\s)}) + \dots
\end{equation*}
\end{lemm}

\begin{proof}
Il suffit de démontrer que le support cuspidal d'un sous-quotient irréductible 
d'une représentation de la forme $\ip_\a\circ\rp_\a(\pi)$, où $\a$ est une 
composition de $\deg(\pi)$, n'est composé que de représentations dans~: 
\begin{equation}
\label{Corinne}
\Om_{\s} \cup \Om_{\st_0(\s)} \cup \Om_{\st_1(\s)} \cup \Om_{\st_2(\s)} \cup \dots 
\end{equation}
Soient $\s' \in \SA$ et $u \geq -1$ tels que $\s=\st_u(\s')$.
Comme le support supercuspidal de $\pi$ est formé de représentations dans 
$\Om_{\s'}$, 
le support cuspidal de $\tau$ est composé de représentations dans~:
\begin{equation*}
\Om_{\s'} \cup \Om_{\st_0(\s')} \cup \Om_{\st_1(\s')} \cup \Om_{\st_2(\s')} \cup \dots 
\end{equation*}
Supposons qu'il y ait dans $\cusp(\tau)$ un terme inertiellement équivalent à 
$\st_i(\s')$, $i\>-1$, tel que $\deg(\st_i(\s')) < \deg(\s)$. 
Fixons une paire $\b=(\b_1,\b_2)$ avec $\b_1=\deg(\st_i(\s'))$ 
telle que $\rp_\b(\tau)$ soit non nulle. 
Si l'on calcule la restriction parabolique 
$\rp_\b(\ip_\a\circ\rp_\a(\pi))$, 
on trouve, grâce au Lemme Géométrique \cite[1.1.2]{MSc} 
dans le cas $p$-adique et à la formule de Mackey 
\cite[\textsection 1.2]{MSf} dans le cas fini, une composition
$\g=(\g_1,\dots,\g_r)$ plus fine que 
$\b$ telle que $\rp_\g(\pi)$ soit non nulle.
En particulier, on a~:
\begin{equation*}
\g_1\<\b_1<\deg(\s).
\end{equation*}
Mais, par hypothèse sur le support cuspidal de $\pi$, la composition 
$\g$ doit être formée de multiples de l'entier $\deg(\s)$, ce 
qui nous donne une contradiction. 
\end{proof}

Le lemme \ref{lemmeclef} peut être reformulé ainsi~:
si $\pi\in\XA_\s^{}$ et $\pi\notin\XA_{\s}^*$,
alors $\AC(\pi)\in\IA_{\s}$. 


\begin{coro}
\label{astablej}
L'algèbre $\RA_\s$ et ses idéaux $\IA_\s$, $\JA_\s$ sont stables par 
l'automorphisme $\AC$.
\end{coro}

\begin{proof}
D'après le théorème \ref{DEC}, toute représentation irréductible $\pi\in\Irr_\s$ 
se décom\-po\-se sous la forme~:
\begin{equation*}
\pi = \pi_{-1}\times\pi_0\times\pi_1\times\pi_2\times\dots
\end{equation*}
où, pour tout $i\>-1$, la représentation $\pi_i$ est irréductible et de 
support cuspidal dans $\Dive(\Om_{\st_i(\s)})$. 
Si l'on applique le lemme \ref{lemmeclef} à $\st_i(\s)$ et $\pi_i$ pour un 
$i\>-1$, on trouve que $\AC(\pi_i)$ appartient à l'idéal $\IA_{\st_i(\s)}$. 
Comme on a~: 
\begin{equation}
\label{decDPI3}
\AC(\pi) = \AC(\pi_{-1})\times\AC(\pi_0)\times\AC(\pi_1)\times\AC(\pi_2)\times\dots
\end{equation}
et comme la famille des $\IA_{\st_i(\s)}$, $i\>-1$, est décroissante, 
on en déduit que $\AC(\pi)$ appartient à $\IA_\s$. 
Comme $\RA_\s$ est égal à $\R\oplus\IA_\s$, on en déduit aussi que 
$\RA_\s$ est stable par $\AC$. 

Si maintenant $\pi\in\Irr_\s\cap\JA_\s$, cela signifie que $\pi\neq\pi_{-1}$. 
Posons~:
\begin{equation*}
\pi_+ = \pi_0\times\pi_1\times\pi_2\times\dots
\in\Irr_{\st_0(\s)}. 
\end{equation*}
On a donc $\pi=\pi_{-1}\times\pi_+$ avec $\AC(\pi_{-1})\in\IA_{\s}$ et 
$\AC(\pi_+)\in\IA_{\st_0(\s)}$.
Comme $\IA_{\st_0(\s)}$ est inclus dans $\JA_\s$, on en déduit que 
$\AC(\pi)\in\JA_\s$. 
\end{proof}

\begin{coro}
\label{RED0}
Soit $\pi\in\Irr_\s$ une représentation irréductible qu'on écrit~:
\begin{equation*}
\pi = \pi_{-1}\times\pi_0\times\pi_1\times\pi_2\times\dots
\end{equation*}
où, pour tout $i\>-1$, la représentation $\pi_i$ est irréductible et de 
support cuspidal dans $\Dive(\Om_{\st_i(\s)})$. 
Si le théorème \ref{ASG} est vrai pour chacun des $\pi_i$ avec $i\>-1$, 
alors il est vrai pour $\pi$. 
\end{coro}

\begin{proof}
Tout terme irréductible $\tau$ de \eqref{decDPI3} apparaît comme 
sous-quotient d'un produit $\tau_1\tdt\tau_r$ où $\tau_i$ est un terme 
irréductible de $\AC(\pi_i)$, pour chaque $i$. 

Supposons que $\tau_i\in\JA_{\st_i(\s)}$ pour au moins un $i$.
Comme $\JA_{\st_i(\s)}$ est inclus dans $\JA_\s$ et que celui-ci 
est un idéal de $\RA_\s$, on en déduit que $\tau\in\JA_\s$. 

Inversement, supposons que $\tau_i\notin\JA_{\st_i(\s)}$ pour tout $i$. 
Par hypothèse, ceci ne se produit que si $\tau_i^{}=\pi_i^\star$ pour tout $i$, 
\ie si et seulement si $\tau$ est égal à~:
\begin{equation*}
\pi^\star=\pi_1^\star \times \dots \times \pi_r^\star.
\end{equation*}
Ceci prouve que le théorème est vrai pour $\pi$.
\end{proof}




\begin{prop}
\label{RED1}
Supposons que, 
pour toute représentation irréductible supercuspidale $\s$,
le théorème \ref{ASG} soit vrai pour toute représentation 
$\pi\in\XA$ dont le support supercuspidal est dans $\Dive(\Om_\s)$. 
Alors le théorème \ref{ASG} est vrai.
\end{prop}

\begin{proof}
Soit $\pi\in\XA$. 
D'après le théorème \ref{DEC1}, il y a des représentations 
irréductibles supercuspidales $\s_1,\dots,\s_r$ non inertiellement équivalentes 
deux à deux et des représentations irré\-duc\-tibles $\pi_1,\dots,\pi_r$ telles que~:
\begin{equation*}
\pi = \pi_{1}\tdt\pi_{r}
\end{equation*}
et $\scusp(\pi_i)\in\Dive(\Om_{\s_i})$ pour tout $i\in\{1,\dots,r\}$. 
On a~:
\begin{equation}
\label{decDPI}
\AC(\pi) = \AC(\pi_{1})\tdt\AC(\pi_r)
\end{equation}
donc tout terme irréductible $\tau$ de \eqref{decDPI} apparaît comme 
sous-quotient d'un produit $\tau_1\tdt\tau_r$ où $\tau_i$ est un terme 
irréductible de $\AC(\pi_i)$. 
Comme $\scusp(\tau_i)=\scusp(\pi_i)$, le théorème \ref{DEC1} implique 
que le produit $\tau_1\tdt\tau_r$ est irréductible, donc égal à $\tau$. 
Ainsi~:
\begin{equation*}
\cusp(\tau)=\cusp(\tau_1)+\dots+\cusp(\tau_r),
\end{equation*}
qui n'est égal à $\cusp(\pi)$ que si $\cusp(\tau_i)=\cusp(\pi_i)$ pour tout 
$i\in\{1,\dots,r\}$. 
Par hypothèse, ceci ne se produit que si $\tau_i^{}=\pi_i^\star$ pour tout $i$, 
\ie si et seulement si $\tau$ est égal à~:
\begin{equation*}
\pi^\star=\pi_1^\star \times \dots \times \pi_r^\star.
\end{equation*}
Ceci prouve que le théorème est vrai pour $\pi$.
\end{proof}


\begin{prop}
\label{RED2}
Supposons que, pour toute représentation supercuspidale $\s$, 
le théo\-rème \ref{ASG} soit vrai pour toute représentation $\pi\in\XA$ 
telle que $\cusp(\pi)\in\Dive(\Om_\s)$. 
Alors le théorème \ref{ASG} est vrai.
\end{prop}

\begin{proof}
Par conjonction du corollaire \ref{RED0} et de la proposition \ref{RED1}.
\end{proof}

Pour la définition de la notation $\FM$ dans l'énoncé suivant, 
on renvoie à \eqref{Falbala}. 


\begin{prop}
\label{RED3}
Soit $\s$ une représentation cuspidale, et soit un entier $n\>1$.
Le théorème \ref{ASG} est vrai pour toute représentation $\pi\in\XA_{\s,n}^{*}$ 
si et seulement si~:
\begin{equation}
\label{mataus}
(-1)^{n}\cdot\FM(\AC(\pi)) \in \MA_{\s}
\end{equation}
est un $\Hh(\s,n)$-module simple.
\end{prop}

\begin{proof}
Rappelons que le théorème \ref{ASG} est vrai pour $\pi$ si et seulement s'il 
existe une repré\-sentation irréductible $\pi^\star$
de même support cuspidal que $\pi$ telle que~: 
\begin{equation*}
\AC(\pi)- (-1)^{r(\pi)} \cdot \pi^\star\in\JA_{\s}.
\end{equation*}
Comme $\pi\in\XA_{\s,n}^{*}$, on a $r(\pi)=n$.
Comme $\JA_{\s}$ est le noyau de $\FM$, cette condition s'écrit~:
\begin{equation*}
(-1)^{n} \cdot \FM(\AC(\pi)) = \FM(\pi^\star).
\end{equation*}
Comme $\FM$ induit une bijection entre $\XA_{\s,n}^{*}$ et les modules à 
droite simples sur $\Hh(\s,n)$ (théorème \ref{qptf}), il s'ensuit que 
le théorème \ref{ASG} est vrai pour $\pi$ si et seulement si \eqref{mataus} 
est un module à droite simple dont l'antécédent par $\FM$ dans 
$\XA_{\s,n}^{*}$ a le même support cuspidal que $\pi$.

Supposons simplement que \eqref{mataus} est un module à droite simple, 
et notons $\pi^\star$ l'unique élément de $\XA_{\s,n}^{*}$ lui correspondant 
par $\FM$. 
Ecrivons $\cusp(\pi)=\s_1+\dots+\s_n$ avec $\s_1,\dots,\s_n\in\Om_\s$.
Alors $\pi^\star$ est un sous-quotient irréductible de $\s_1\tdt\s_n$ dont 
le support cuspidal est formé de termes inertiellement équivalents à $\s$. 
On a donc $\cusp(\pi^\star)=\s_1+\dots+\s_n=\cusp(\pi)$, ce qui met fin à la 
démonstration. 
\end{proof}

\begin{rema}
Dans le cas où $\F$ est fini et $\s$ est le caractère trivial de $\mult\F$, 
la proposition \ref{RED2} et le lien avec l'algèbre de Hecke 
sont prouvés par Ackermann et Schroll (voir \cite{AS}, notamment le théorème 4.1).
\end{rema}

\subsection{Conclusion partielle}
\label{ConcPart}

Si l'on joint les propositions \ref{RED2} et \ref{RED3}, 
on voit que, pour prouver le théorème \ref{ASG}, il suffit de prouver 
que, pour toute représentation supercuspidale $\s\in\SA$, tout $n\>1$
et toute repré\-sen\-tation irréductible $\pi\in\XA_{\s,n}^{*}$, la quantité~:
\begin{equation*}
(-1)^{n}\cdot\FM(\AC(\pi)) 
\end{equation*}
dans $\MA_{\s}$ est un $\Hh(\s,n)$-module simple.


\section{Un calcul de coinvariants}
\label{SEC3}

On fixe une représentation irréduc\-ti\-ble cuspidale $\s$ 
de degré $m$ et un entier $n\>1$, et on pose $\Q=\Si^{\times n}$
et $\Hh=\Hh(\s,n)$ (voir le paragraphe \ref{hecke} pour les notations).

On fixe une composition $(n_1,\dots,n_r)$ de $n$, et on pose~:
\begin{equation*}
\M = \GL_{mn_1}(\D)\tdt\GL_{mn_r}(\D) \subseteq \GL_{mn}(\D) = \G.
\end{equation*}
On note $\P$ le sous-groupe parabolique de $\G$ en\-gen\-dré par 
$\M$ et les matrices triangu\-laires supé\-rieu\-res, et on note $\N$ 
son radical unipotent. 
On pose $\Q_\M=\Si^{\times n_1}\odo\Si^{\times n_r}$,
et on note $\Hh_\M$ son algèbre d'endomorphismes. 

La représentation $\Q$ s'identifie à l'induite parabolique de $\Q_\M$ à $\G$ 
le long de $\P$. 
Par fonctoria\-lité, on en déduit un morphisme (injectif) d'algèbres~:
\begin{equation*}
\label{HMH}
j : \Hh_\M \to \Hh
\end{equation*}
faisant de $\Hh$ un $\Hh_\M$-module à droite. 

\subsection{Le cas fini}
\label{casfiniSEC3}

Dans ce paragraphe, on suppose qu'on est dans le cas fini. 
Comme $\N$ est un $p$-groupe fini et que $p$ est inversible dans $\R$, 
les $\N$-coinvariants $\Q_\N$ sont canoniquement isomorphes aux 
$\N$-invariants $\Q^\N$.

\begin{prop}
\label{isoQN}
On a un isomorphisme~:
\begin{equation*}
\Q^\N \simeq \Hh \otimes_{\Hh_{\M}} \Q_{\M}
\end{equation*}
de représentations de $\M$ et de $\Hh$-modules à droite.
\end{prop}

\begin{proof}
On note $i$ le plongement canonique de $\Q_\M$ dans $\Q$, 
défini, pour tout $f\in\Q_\M$, par~:
\begin{equation*}
i(f) : g \mapsto 
\left\{
\begin{array}{ll}
f(m) & \text{si $g=mn$ avec $m\in\M$, $n\in\N$}, \\
0 & \text{si $g\notin\P$},
\end{array}
\right.
\end{equation*}
C'est un morphisme injectif de représentations de $\M$, 
tel que $i\circ k=j(k)\circ i$ pour tout $k\in\Hh_\M$. 
Son image est le sous-espace des fonctions de support inclus dans $\P$~; 
elle est invariante par $\N$. 
On note $\xi$ l'application~:
\begin{eqnarray*}
\Hh \otimes_{\Hh_{\M}} \Q_{\M} & \to & \Q \\
h\otimes f & \mapsto & h(i(f)).
\end{eqnarray*}
On note $\W$ le sous-groupe des matrices de permutation de $\GL_{n}(\F)$ 
(naturellement plongé dans $\G$) et $\SS$ l'ensemble des matrices des 
transpositions $i \leftrightarrow i+1$. 
On note $\W_\M$ le sous-groupe de $\W$ correspondant à $\M$. 

Soit $\Dd$ l'ensemble des représentants distingués de $\W/\W_\M$ dans $\W$, 
\ie que, pour tout $d\in\Dd$, l'unique élément de longueur minimale dans
$d\W_\M$ est $d$.
Fixons pour tout élément $d\in\Dd$
un élément $t_{d}\in\Hh$ de support $\P d\P$. 
Alors $\Hh$ est un $\Hh_\M$-module à droite libre de base $(t_d)_{d\in\Dd}$. 
Le morphisme $t_{d}\circ i$ est injectif, et il a pour image le sous-espace de 
$\Q^{\N}$ formé des fonctions de support inclus dans $\P d\N$. 
Ainsi $\xi$ est bijectif. 

Pour plus de détails, on pourra consulter \cite{OSf}, paragraphes 4.2 et 4.3.
\end{proof}

Appliquant le foncteur d'induction parabolique de $\M$ à $\G$ le long de 
$\P$, on obtient un iso\-mor\-phisme~:
\begin{equation}
\label{Coro0}
\Ind_{\P}^{\G}(\Q^\N) \simeq 
\Hh \otimes_{\Hh_{\M}} \Q
\end{equation}
de représentations de $\G$. 

\begin{coro}
\label{Coro1}
\begin{enumerate}
\item
Pour toute représentation $\varrho$ de $\M$, on a un isomorphisme~:
\begin{equation*}
\Hom_{\G}(\Q,\Ind_{\P}^{\G}(\varrho)) \simeq 
\Hom_{\Hh_{\M}}(\Hh,\Hom_{\M}(\Q_{\M},\varrho))
\end{equation*}
de $\Hh$-modules à droite. 
\item
Le foncteur exact $\pi\mapsto\Ind_{\P}^{\G}(\pi^\N)$ préserve la sous-catégorie 
$\Ee(\s,n)$ (voir le \S{\rm\ref{hecke}}). 
\end{enumerate}
\end{coro}

\begin{proof}
On obtient (1) par adjonction, grâce à 
la proposition \ref{isoQN}. 
Pour obtenir (2), il suffit prouver que l'image de $\Q$ par ce foncteur 
est dans $\Ee(\s,n)$, ce qui suit de \eqref{Coro0}.
\end{proof}

\subsection{Le cas $p$-adique}
\label{sec:Types}

Dans ce paragraphe, on suppose qu'on est dans le cas $p$-adique. 
On fixe, comme au paragraphe \ref{hecke}, 
un type simple maximal $(\J,\l)$ contenu dans $\s$. 
Ce type simple maximal admet une décompo\-si\-tion (non canonique)~: 
\begin{equation*}
\l=\k\otimes\s^{{\rm fin}}
\end{equation*}
où $\k$ est une $\b$-extension au sens de \cite[\S2.4]{MSt}
et $\s^{{\rm fin}}$ une représentation irréductible de $\J$ triviale sur un 
sous-groupe ouvert distingué $\J^1$ de $\J$.
Le quotient $\J/\J^1$ s'identifie (non canoniquement) à un groupe fini
$\GL_{f}(\FB)$
où $\FB$ est un corps fini de caractéristique $p$ et $f$ un entier, 
et $\s^{{\rm fin}}$ s'identifie à une représentation irréductible cuspidale de 
$\GL_{f}(\FB)$.

On note $\GB$ le groupe $\GL_{fn}(\FB)$.
L'induite parabolique $\s^{{\rm fin}}\tdt\s^{{\rm fin}}$ est une représentation de $\GB$ 
notée $\Q^{{\rm fin}}$ et son algèbre d'endomorphismes est notée 
$\Hh^{{\rm fin}}$. 
Comme au début de la section \ref{SEC3}, on a aussi une 
représentation $\Q^{{\rm fin}}_\M$ du sous-groupe de Levi 
$\MB=\GL_{fn_1}(\FB)\tdt\GL_{fn_r}(\FB)$ et une sous-algèbre 
$\Hh^{{\rm fin}}_{\M}$ de $\Hh^{{\rm fin}}$.
D'après le corollaire \ref{Coro1}, on a un isomorphisme~:
\begin{equation*}
\Hom_{\GB}(\Q^{{\rm fin}},\Ind_{\PB}^{\GB}(\varrho)) \simeq 
\Hom_{\Hh^{{\rm fin}}_{\M}}(\Hh^{{\rm fin}},\Hom_{\M}(\Q^{{\rm fin}}_{\M},\varrho))
\end{equation*}
de $\Hh^{{\rm fin}}$-modules à droite,
pour toute représentation $\varrho$ de $\MB$.

D'après \cite{MSt}, sections 2 et 5, il existe~:
\begin{enumerate}
\item
un sous-groupe ouvert compact $\BJ$ de $\G$ et un sous-groupe 
ouvert distingué $\BJ^{1}$ tels que le 
quotient $\BJ/\BJ^{1}$ s'identifie au groupe $\GB$~; 
\item
une représentation irréductible $\bk$ de $\BJ$ telle qu'on ait un isomorphisme
de représentations ~:
\begin{equation}
\label{Nicolas}
\Q \simeq \ind^{\G}_{\BJ}(\bk \otimes \Q^{{\rm fin}}) 
\end{equation}
où $\Q^{{\rm fin}}$ est vue, par inflation, comme une représentation de 
$\BJ$ triviale sur $\BJ^{1}$.
\end{enumerate}
(Dans la section 5 de \cite{MSt}, ces groupes et cette représentation sont 
notés $\BJ_{{\rm max}}^{}$, $\BJ_{{\rm max}}^1$ et $\bk_{{\rm max}}^{}$.)

L'isomorphisme \eqref{Nicolas} induit un morphisme fonctoriel d'algèbres 
de $\Hh^{{\rm fin}}$ dans $\Hh$.
Si l'on définit le foncteur exact~:
\begin{equation*}
\KM : \pi \mapsto \Hom_{\BJ^{1}}(\bk,\pi)
\end{equation*}
de la catégorie des représentations de $\G$ vers celle des représentations de 
$\GB$, 
on a un iso\-morphis\-me de $\Hh^{{\rm fin}}$-modules à droite~: 
\begin{equation}
\label{Lemme1bis}
\Hom_{\G}(\Q,\pi) \simeq \Hom_{\GB}(\Q^{{\rm fin}},\KM(\pi))
\end{equation}
pour toute repré\-sen\-ta\-tion lisse $\pi$ de $\G$.

De façon analogue, il y a un foncteur $\KM_\M$ 
de la catégorie des représentations de $\M$ vers celle des représentations de 
$\MB$ possédant les propriétés suivantes~:
\begin{enumerate}
\item
on a un iso\-morphis\-me de $\Hh_\M^{{\rm fin}}$-modules à droite~: 
\begin{equation}
\label{Lemme1ter}
\Hom_{\M}(\Q_\M,\varrho) \simeq \Hom_{\MB}(\Q_\M^{{\rm fin}},\KM_\M(\varrho))
\end{equation}
pour toute repré\-sen\-ta\-tion lisse $\varrho$ de $\M$~;
\item
si $\PB$ est le sous-groupe parabolique standard de $\GB$ 
correspondant au sous-groupe de Levi $\MB$, 
alors pour toute repré\-sen\-ta\-tion $\varrho$ de $\M$
on a (\cite[Propo\-sition 5.6]{SSb}) un isomorphisme~:
\begin{equation}
\label{Cagliostro}
\KM(\Ind_\P^\G(\varrho))\simeq\Ind_{\PB}^{\GB}(\KM_\M(\varrho))
\end{equation}
de représentations de $\GB$.
\end{enumerate}

\begin{prop}
\label{JCRomand}
On a un isomorphisme~:
\begin{equation*}
\Q_\N \simeq 
\Hh \otimes_{\Hh_{\M}} \Q_{\M}
\end{equation*}
de $\Hh$-modules à gauche et de représentations de $\M$. 
\end{prop}

\begin{proof}
Appliquons \eqref{Lemme1bis} à la représentation 
$\Ind_\P^\G(\varrho)$ où $\varrho$ est une 
$\R$-repré\-sen\-ta\-tion lisse de $\M$. 
Compte tenu de \eqref{Cagliostro}, on a des isomorphismes 
de $\Hh^{{\rm fin}}$-modules à droite~: 
\begin{equation*}
\Hom_{\G}(\Q,\Ind_\P^\G(\varrho)) \simeq 
\Hom_{\GB}(\Q^{{\rm fin}},\KM(\Ind_\P^\G(\varrho))) 
\simeq \Hom_{\GB}(\Q^{{\rm fin}},\Ind_{\PB}^{\GB}(\KM_\M(\varrho))).
\end{equation*}
D'après le corollaire \ref{Coro1} et \eqref{Lemme1ter}, 
on a des isomorphismes 
de $\Hh^{{\rm fin}}$-modules à droite~:
\begin{eqnarray*}
\Hom_{\GB}(\Q^{{\rm fin}},\Ind_{\PB}^{\GB}(\KM_\M(\varrho))) 
& \simeq &
\Hom_{\Hh^{{\rm fin}}_{\M}}(\Hh^{{\rm fin}},
\Hom_{\MB}(\Q^{{\rm fin}}_\M,\KM_\M(\varrho))) \\
& \simeq & 
\Hom_{\Hh^{{\rm fin}}_{\M}}(\Hh^{{\rm fin}},\Hom_{\M}(\Q_{\M},\varrho))
\end{eqnarray*}
et ce dernier est isomorphe à 
$\Hom_{\M}(\Hh^{{\rm fin}}\otimes_{\Hh^{{\rm fin}}_{\M}}\Q_{\M},\varrho)$.
Ceci étant valable pour toute représen\-ta\-tion lisse $\varrho$ de 
$\M$, on en déduit un isomorphisme~:
\begin{equation}
\label{attente}
\Q_\N \simeq 
\Hh^{{\rm fin}} \otimes_{\Hh^{{\rm fin}}_{\M}} \Q_{\M}
\end{equation}
de $\Hh^{{\rm fin}}$-modules à gauche et de représentations de $\M$. 

Avec les notations du paragraphe \ref{hecke}, 
et compte tenu de l'isomorphisme \eqref{Hodkann}, 
le $\Hh^{{\rm fin}}_{\M}$-module à gauche $\Hh_\M$ est libre de base 
$(\X^\a)_{\a\in\ZZ^n}$ avec $\X^\a=\X_1^{\a_1}\dots\X_n^{\a_n}$ pour 
$\a=(\a_1,\dots,\a_n)\in\ZZ^n$, et c'est aussi une base du 
$\Hh^{{\rm fin}}$-module à droite $\Hh$.
Ainsi l'application naturelle~:
\begin{eqnarray*}
\Hh^{{\rm fin}} \otimes_{\Hh^{{\rm fin}}_{\M}} \Hh_\M & \to & \Hh \\
(h,k) & \mapsto & h*k
\end{eqnarray*}
(où $*$ désigne la multiplication dans $\Hh$) 
est un isomorphisme de $(\Hh^{{\rm fin}},\Hh_\M)$-bimodules.
Comme $\Q_\M$ est un $\Hh_\M$-module à gauche, 
on obtient, grâce à \eqref{attente}, l'isomorphisme annoncé.
\end{proof}

Comme dans le cas fini, on en déduit le corollaire suivant.

\begin{coro}
\label{stableEsn}
\begin{enumerate}
\item
Pour toute représentation $\varrho$ de $\M$, on a un isomorphisme~:
\begin{equation*}
\Hom_{\G}(\Q,\Ind_{\P}^{\G}(\varrho)) \simeq 
\Hom_{\Hh_{\M}}(\Hh,\Hom_{\M}(\Q_{\M},\varrho))
\end{equation*}
de $\Hh$-modules à droite. 
\item
Le foncteur exact $\pi\mapsto\Ind_{\P}^{\G}(\pi_\N)$ préserve la sous-catégorie 
$\Ee(\s,n)$ (voir le \S{\rm\ref{hecke}}). 
\end{enumerate}
\end{coro}

\begin{rema}
\label{vaval}
D'après \cite{MSt}, la représentation $\Q$ s'identifie à 
l'induite compacte d'un type semi-simple (\textit{ibid.}, \S2.8-2.9).
D'après la propriété de paire couvrante de ces types  semi-simples 
(\textit{ibid.}, \S2.7), 
on a un isomorphisme de $\Hh_\M$-modules à droite~:
\begin{equation*}
\Hom_{\G}(\Q,\pi) \simeq \Hom_{\M}(\Q_{\M},\pi_\N)
\end{equation*}
pour toute représentation $\pi$ de $\G$, 
où $\pi_\N$ désigne la restriction parabolique (normalisée) de $\pi$.
\end{rema}



\subsection{Un corollaire}

On suppose à nouveau qu'on est, indistinctement, dans le cas fini ou $p$-adique. 
Des corollaires \ref{Coro1} et \ref{stableEsn} on déduit le résultat suivant. 

\begin{coro}
\label{Tet}
Le morphisme de groupes $\FM:\RA_\s\to\MA_\s$ 
défini en \eqref{Falbala} est un morphisme d'algèbres.
\end{coro}

Si l'on note $\AA_\s$ la $\ZZ$-algèbre quotient de $\RA_\s$ par $\JA_\s$, 
le morphisme $\FM$ induit un isomorphisme~: 
\begin{equation}
\label{isoAAfMA}
\fm : \AA_\s \to \MA_\s
\end{equation}
de $\ZZ$-algèbres.
D'après le corollaire \ref{astablej}, l'involution $\AC$ induit une involution 
sur $\AA_\s$, notée $\ac$.
On a défini une involution $\inv$ 
sur $\MA_\s$ au paragraphe \ref{hecke}.
Nous étudions maintenant le comportement de $\fm$
vis-à-vis de ces deux invo\-lu\-tions.
Plus précisément, 
nous avons deux isomorphismes d'algèbres $\fm\circ\ac$ et $\bt\circ\fm$ de 
$\AA_\s$ dans $\MA_\s$, que nous allons comparer. 

\subsection{Une formule pour $\FM\circ\AC$}
\label{FormuleRaviolisVapeur}


\begin{lemm}
\label{stableEsn2}
Soit $\a$ une composition de $mn$.
\begin{enumerate}
\item
Si $\a$ n'est pas de la forme $(mn_1,\dots,mn_r)$ où $(n_1,\dots,n_r)$ est une 
composition de $n$, 
alors le foncteur $\ip_\a\circ\rp_\a$ est nul sur la sous-catégorie $\Ee(\s,n)$.
\item
Dans tous les cas, le foncteur $\ip_\a\circ\rp_\a$ 
préserve la sous-catégorie $\Ee(\s,n)$.
\end{enumerate}
\end{lemm}

\begin{proof}
Si $\a$ est de la forme $(mn_1,\dots,mn_r)$, c'est une conséquence des 
corollaires \ref{Coro1} et \ref{stableEsn}.
Sinon, la restriction de $\ip_\a\circ\rp_\a$ à $\Ee(\s,n)$ est le foncteur 
nul, car le support cuspidal d'une représentation irréductible dans 
$\Ee(\s,n)$ n'est composé que de représentations dans \eqref{Corinne}.
\end{proof}

Rappelons 
que, pour toute composition $\g=(n_1,\dots,n_r)$ de $n$, 
on a défini des foncteurs 
${\bf r}_\g$ et ${\bf i}_\g$ au paragraphe \ref{hecke}. 
Leur composé ${\bf i}_\g\circ{\bf r}_\g$ 
définit donc un endomorphisme de groupe de $\MA_{\s}$.

Ecrivons $m\cdot\g=(mn_1,\dots,mn_r)$, qui est une composition de $mn$. 

\begin{lemm}
\label{calculdelta}
\begin{enumerate}
\item
Pour toute composition $\g$ de $n$
et toute représentation $\pi\in\XA_{\s,n}$, on a~:
\begin{equation*}
\FM(\ip_{m\cdot\g}\circ\rp_{m\cdot\g}(\pi)) \simeq {\bf i}_\g\circ{\bf r}_\g(\FM(\pi))
\end{equation*}
dans la catégorie des $\Hh(\s,n)$-modules à droite. 
\item
Pour toute représentation $\pi\in\XA_{\s,n}$, on a~:
\begin{equation*}
\FM(\AC(\pi)) = \sum\limits_{\g}
(-1)^{r(\g)}\cdot{\bf i}_\g\circ{\bf r}_\g(\FM(\pi))
\end{equation*}
dans $\MA_\s$, où $\g$ décrit les compositions de $n$.
\end{enumerate}
\end{lemm}

\begin{proof}
La partie (2) est une conséquence du lemme \ref{stableEsn2} et de la partie (1). 

D'après le corollaire \ref{stableEsn}, on a un isomorphisme
de $\Hh$-modules à droite~:
\begin{equation*}
\Hom_{\G}(\Q,\ip_{m\cdot\g}(\varrho)) \simeq 
 {\bf i}_\g(\Hom_{\M}(\Q_{\M},\varrho))
\end{equation*}
pour toute représentation $\varrho$ de $\M$.
D'après la remarque \ref{vaval}, on a également un isomorphisme
de $\Hh_\M$-modules à droite~:
\begin{equation}
\label{Gilles}
{\bf r}_\g(\Hom_{\G}(\Q,\pi)) \simeq \Hom_{\M}(\Q_{\M},\rp_{m\cdot\g}(\pi))
\end{equation}
pour toute représentation $\pi$ de $\G$. 
En composant les deux, on obtient l'identité voulue.
\end{proof}



\section{Preuve du théorème \ref{ASG}}
\label{four}

Les notations $\s$ et $m,n$ ont le même sens qu'au début de la section \ref{SEC3}.

\subsection{Changement de groupe}
\label{GreatExpectations}

On permet maintenant de changer le corps $\F$ fixé dans l'introduction~:
fixons un corps $\F'$ de même nature que $\F$,
\ie fini de caractéristique $p$ dans le cas fini, 
et localement compact non archimédien de caractéristique résiduelle $p$ 
dans le cas $p$-adique.
On fixe une représentation irréductible cuspidale $\s'$ d'un groupe 
$\GL_{m'}(\D')$, où $\D'$ est une $\F'$-algèbre à division centrale et 
$m'\>1$, telle que $q_{\s'}=\qr$ dans $\R$.
Comme en \eqref{Falbala}, on associe à $\s'$ un morphisme d'algèbres~:
\begin{equation*}
\FM' : \RA_{\s'} \to \MA_{\s'}.
\end{equation*}
Cette dernière est égale à $\MA_\s$ 
(que l'on note donc $\MA$) puisque $q_{\s'}=\qr$. 
Le morphisme $\FM'$ induit une bijection entre $\XA_{\s'}^*$ et l'ensemble des 
objets simples de $\MA$.
On note $\AC'$ l'involution sur $\RA_{\s'}$. 

D'après \cite{MSt} (voir aussi \cite[\S5.2]{MSc}), 
il y a pour tout entier $n\>0$ une bijection~:
\begin{equation*}
\boldsymbol{\Phi}_{n}^{} : \XA_{\s,n}^* \to \XA_{\s',n}^*
\end{equation*}
compatible au support cuspidal. 
En prenant la réunion sur $n$, 
on obtient une bijection $\boldsymbol{\Phi}$ de 
$\XA_{\s}^*$ dans $\XA_{\s'}^*$.
Pour toute représentation irréductible $\pi\in\XA_{\s}^*$, 
on a l'identité $\FM'(\boldsymbol{\Phi}(\pi))=\FM(\pi)$.

\begin{prop}
\label{ChGrFD}
Etant donnée $\pi\in\XA_{\s}^*$, 
notons $\pi'$ son image par $\boldsymbol{\Phi}$.
On a l'identité~: 
\begin{equation*}
\FM'(\AC'(\pi'))=\FM(\AC(\pi)).
\end{equation*}
\end{prop}

\begin{proof}
D'après le lemme \ref{calculdelta}, on a~:
\begin{eqnarray*}
\FM'(\AC'(\pi')) & = & \sum\limits_{\g}
(-1)^{r(\g)}\cdot{\bf i}_\g\circ
{\bf r}_\g(\FM'(\pi')) \\
& = & \sum\limits_{\g}
(-1)^{r(\g)}\cdot{\bf i}_\g\circ
{\bf r}_\g(\FM(\pi)) 
\end{eqnarray*}
(où $\g$ décrit les compositions de $n$)
et cette dernière somme est égale à $\FM(\AC(\pi))$.
\end{proof}

\begin{coro}
\label{COROCHGTGP}
Si le théorème \ref{ASG} est vrai pour toute représentation 
$\pi'\in\XA_{\s'}^*$, 
alors il est vrai pour toute représentation $\pi\in\XA_{\s}^*$.
\end{coro}

\begin{proof}
Soit $\pi\in\XA_{\s,n}^{*}$ pour un $n\>1$, 
et posons $\pi'=\boldsymbol{\Phi}(\pi)$, 
qui appartient à $\XA_{\s',n}^*$. 
D'après la proposition \ref{ChGrFD}, on a~:
\begin{equation*}
(-1)^{n}\cdot\FM'(\AC'(\pi')) = (-1)^{n}\cdot\FM(\AC(\pi)).
\end{equation*}
Par hypothèse, le théorème \ref{ASG} est vrai pour $\pi'$, 
donc le membre de gauche est un module simple dans $\MA$ 
d'après la proposition \ref{RED3}.
Le membre droite l'est donc aussi, et il s'ensuit 
(à nouveau grâce à la proposition \ref{RED3})
que le théorème \ref{ASG} est vrai pour $\pi$.
\end{proof}

\subsection{La théorie des segments}
\label{segments}
\label{ZL}
\label{clasZ}

On rappelle maintenant la notion de segment introduite
dans \cite{MSc}. 
Rappelons que nous avons défini au paragraphe \ref{DefRepCusp},
pour toute représentation cuspidale $\s$, 
un ca\-rac\-tère $\nr^{}$ (voir aussi \cite[§4.5]{MSt}).



\begin{defi}
\label{DEFSEG}
Un {\it segment} est un couple $[\s,n]$ formé d'une classe 
de représentation irré\-duc\-tible cuspidale $\s\in\CA$ et 
d'un entier $n \geq 1$. 
\end{defi}


Soit un segment $[\s,n]$. 
On note $\EuScript{Z}(\s,n)$ le caractère de $\Hh(\s,n)$ défini par~:
\begin{equation*}
\SS_i\mapsto \qr,\ i\in\{1,\dots,n-1\},
\quad
\X_j^{}\mapsto \qr^{j-1},\ j\in\{1,\dots,n\},
\end{equation*}
et on note $\EuScript{L}(\s,n)$ le caractère de $\Hh(\s,n)$ défini par~:
\begin{equation*}
\SS_i\mapsto-1,\ i\in\{1,\dots,n-1\},
\quad 
\X_j^{}\mapsto \qr^{n-j},\ j\in\{1,\dots,n\}.
\end{equation*}
On remarque que ces deux caractères sont échangés par l'involution $\bt$ 
définie au paragraphe \ref{hecke}.

On associe au segment $[\s,n]$ deux représentations irréductibles de $\G_{mn}$.

\begin{defi}[{\cite[\S7.2]{MSc}, \cite[\textsection 3]{MSf}}]
Soit un segment $[\s,n]$. 
\begin{enumerate}
\item
On note $\Z(\s,n)$
l'unique sous-représentation irréductible de $\s\times\s\nr^{}\tdt\s\nr^{n-1}$ 
corres\-pondant par \eqref{BIJ} au caractère $\EuScript{Z}(\s,n)$.
\item
On note $\L(\s,n)$
l'unique quotient irréductible de $\s\times\s\nr^{}\tdt\s\nr^{n-1}$ 
correspondant par \eqref{BIJ} au caractère $\EuScript{L}(\s,n)$.
\end{enumerate}
\end{defi}

Nous aurons besoin de la propriété suivante. 

\begin{theo}[{\cite[Lemme 9.41]{MSc}, \cite[Lemme 4.8]{MSf}}]
\label{basestd}
Supposons que $\s$ est supercuspidale,
et notons $\Dd_{\s}$ l'en\-sem\-ble des segments 
$[\s',n]$ tels que $\s'\in\Om_{\s}$ et $n\>1$.
Les produits~: 
\begin{equation}
\label{defRSTD}
\Z(\s_1,n_1)\tdt\Z(\s_r,n_r),
\end{equation}
lorsque $[\s_1,n_1]+\dots+[\s_r,n_r]$ décrit 
$\Dive(\Dd_{\s})$, 
forment une base du $\ZZ$-module libre $\RA_{\s}$.
\end{theo}

Pour comparer les morphismes d'algèbres $\fm\circ\ac$ et $\bt\circ\fm$
il suffit de les comparer sur une base de $\AA_\s$. 
D'après le théorème \ref{basestd}, il suffit 
de le faire pour les représentations $\Z(\s',n)$ avec 
$[\s',n]\in\Dd_{\s}$ et $\s$ super\-cuspidale.

\begin{rema}
\label{ZeleNonInv}
Grâce au théorème \ref{basestd} (voir aussi la remarque A.6 de \cite{MSb}), 
il y a un unique automorphisme de $\ZZ$-algèbre $\EC$ de $\RA_{\s}$ 
tel que~:
\begin{equation*}
\EC(\Z(\s',n))=\L(\s',n)
\quad\text{pour tout}\quad
[\s',n]\in\Dd_{\s}.
\end{equation*}
Dans le cas banal, $\EC$ est involutif et coïncide à un signe près avec 
l'involution d'Aubert (ap\-pen\-di\-ce de \cite{MSb}).
Dans le cas non banal, $\EC$ n'est pas toujours involutif, comme 
le montre l'exem\-ple suivant. 
Supposons que la caractéristique $\ell$ de $\R$ divise $q^2+q+1$ et 
que $\s$ est le caractère trivial de $\mult\F$, noté $1$.
Alors on a~:
\begin{eqnarray*}
\EC(\Z(1,3)) &=& \L(1,3), \\
\EC(\L(1,3)) &=& \Z(1,3)+\st_0(1), \\
\EC(\st_0(1)) &=& -2\cdot\st_0(1),
\end{eqnarray*}
ce qui montre que $\EC$ n'est pas involutif, ne préserve pas l'irréductibilité 
à un signe près et n'est pas égal à $\AC$.
\end{rema}

\subsection{Une formule pour $\AC(\Z(\s,n))$}
\label{FormuleDZ}

Dans ce paragraphe et le suivant, $\R$ est 
une clôture algébrique $\flb$ d'un corps fini de caractéris\-tique 
$\ell\neq p$. 
On fixe une clôture algébrique $\qlb$ du corps des nombres 
$\ell$-adiques.
On note~: 
\begin{equation*}
\r_\ell : \widetilde{\RA}^{{\rm ent}} \to \RA
\end{equation*}
le morphisme de réduc\-tion modulo $\ell$, où 
$\widetilde{\RA}$ désigne le groupe abélien libre engendré 
par les clas\-ses de $\qlb$-représentations irré\-ductibles des $\G_m$, $m\>0$
et $\widetilde{\RA}^{{\rm ent}}$ le sous-groupe de 
$\widetilde{\RA}$ engendré par les classes de 
représentations irréductibles entières (voir \cite[\S1.2]{MSc}). 

\begin{lemm}
\label{Drl}
Soit $\s$ une $\flb$-représentation irréductible cuspidale.
Supposons qu'il existe une $\qlb$-représentation irréductible 
cuspidale entière $\widetilde{\s}$ dont la réduction mod $\ell$ 
est égale à $\s$. 
Alors~:
\begin{equation*}
\AC(\Z(\s,n)) = (-1)^n\cdot\r_{\ell}(\L(\widetilde{\s},n))
\end{equation*}
pour tout $n\>1$.
\end{lemm}

\begin{proof}
Notons $\widetilde{\AC}$ l'involution sur $\widetilde{\RA}$.
Comme $\r_\ell$ com\-mute à l'induction et à la restric\-tion 
paraboliques (voir \cite[\S1.2]{MSc}), on a l'égalité~: 
\begin{equation*}
\label{descenso}
\AC\circ\r_{\ell}= \r_{\ell}\circ\widetilde{\AC}.
\end{equation*}
D'après \cite[Théorème 9.39]{MSc} dans le cas $p$-adique 
et \cite[Lemme 5.9]{MSf} dans le cas fini, 
la réduction mod $\ell$ de $\Z(\widetilde{\s},n)$ est égale à 
$\Z(\s,n)$.
En réalité, les preuves de ces deux résultats comportent une 
omission, que nous allons corriger ici. 
Ces preuves montrent toutes les deux que la réduction mod $\ell$ de 
$\Z(\widetilde{\s},n)$ est irréductible (notons-la $\pi$) et que, 
pour toute composition $\a=(mk,m(n-k))$ avec $m=\deg(\s)$ et 
$k\in\{1,\dots,n-1\}$, on a 
$\rp_{\a}(\pi)=\Z(\s,k)\otimes\Z(\s\nu_{\s}^k,m-k)$.
Mais ceci ne suffit pas (lorsque $q_\s$ est congru à $1$ modulo $\ell$)
pour en déduire que $\pi$ est égale à $\Z(\s,n)$.
Expliquons, dans le cas fini, pourquoi on a cette égalité~; 
l'argument est similaire dans le cas $p$-adique. 

Soit $\L$ un $\zlb$-réseau de $\Z(\widetilde{\s},n)$ stable par $\G$, 
et soit $\L_{\widetilde{\s}}$ un $\zlb$-réseau de $\widetilde{\s}$ 
stable par $\G_{m}$.
La réduction de $\L_{\widetilde{\s}}$ étant isomorphe à $\s$ et celle de $\L$ 
étant isomorphe à $\pi$, on obtient un morphisme naturel~: 
\begin{equation*}
\Hom_{\G}(\L_{\widetilde{\s}}^{\times n},\L)
\otimes\flb \to \Hom_{\G}(\s^{\times n},\pi)
\end{equation*}
de $\flb$-espaces vectoriels.
C'est même un morphisme de $\Hh(\s,n)$-modules, si le membre de gauche 
est muni d'une structure de $\Hh(\s,n)$-modules grâce à l'isomorphisme 
naturel d'algèbres~:
\begin{equation*}
\End_\G(\L_{\widetilde{\s}}^{\times n})\otimes\flb \to \Hh(\s,n).
\end{equation*}
On voit ainsi que $\FM(\pi)$ contient la réduction mod $\ell$ 
du caractère $\EuScript{Z}(\widetilde{\s},n)$ de $\Hh(\widetilde{\s},n)$,
qui est égal à $\EuScript{Z}(\s,n)$.
On en déduit que $\pi$ contient (donc est isomorphe à) $\Z(\s,n)$.

Reprenons la preuve du lemme \ref{Drl}.
D'après ce qu'on sait en caractéristique nulle, 
$\widetilde{\AC}(\Z(\widetilde{\s},n))$ est égal à 
$(-1)^{n}\cdot\L(\widetilde{\s},n)$.
On en déduit le résultat voulu en appliquant $\r_\ell$. 
\end{proof}


\subsection{Le cas unipotent}
\label{CasUnip}

Dans ce paragraphe, on suppose que $\s$ est un caractère non ramifié de 
$\mult\F$, noté $\chi$.
Pour tout $n\>1$, la représentation $\Z(\chi,n)$ est donc le 
$\flb$-caractère non ramifié $\chi\circ\det$ de $\GL_n(\F)$. 

\begin{lemm}
\label{Lemmeunipotent}
Pour tout $n\>1$, on a $\FM(\AC(\Z(\chi,n)))=(-1)^n\cdot\EuScript{L}(\chi,n)$. 
\end{lemm}

\begin{proof}
Quitte à tordre par $\chi$, on peut supposer que $\chi$ est le caractère trivial 
de $\mult\F$, noté $1$.
Soit $\V$ la $\qlb$-représentation de Steinberg de $\G$.
D'après le lemme \ref{Drl}, il suffit de montrer que $\FM(\r_\ell(\V))$ est 
égal à $\EuScript{L}(1,n)$.
Fixons un $\zlb$-réseau $\L$ de $\V$ stable par $\G$.
On a une suite exacte de $\zlb\G$-modules~:
\begin{equation}
\label{SECML}
0 \to \mathfrak{p}\L \to \L \to \overline{\L} = \L\otimes\flb \to 0
\end{equation}
où $\mathfrak{p}$ désigne l'idéal maximal de $\zlb$.
Notons $\I$ le sous-groupe d'Iwahori standard de $\G$,
et notons $\I_1$ son radical pro-unipotent, qui est un pro-$p$-sous-groupe 
distingué de $\I$. 
Ainsi $\FM$ est le foncteur des vecteurs 
$\I$-invariants. 
En l'appliquant à \eqref{SECML}, on trouve une suite exacte~: 
\begin{equation*}
0 \to (\mathfrak{p}\L)^{\I} \to\L^\I \to \overline{\L}^\I \to \H^1(\I,\mathfrak{p}\L)
\end{equation*}
où $\H^1(\I,\mathfrak{p}\L)$ désigne le premier groupe de cohomologie continue 
de $\I$ à coefficients dans $\mathfrak{p}\L$.
On va montrer qu'il est nul.
Comme $\I_1$ est un pro-$p$-sous-groupe distingué de $\I$,
la suite d'inflation-restriction induit un isomorphisme~:
\begin{equation*}
\H^1(\I,\mathfrak{p}\L) \simeq \H^1(\I/\I_1,(\mathfrak{p}\L)^{\I_1})
\end{equation*}
de $\zlb$-modules. 
Appliquant le foncteur exact des $\I_1$-invariants à \eqref{SECML}, on 
trouve que $(\mathfrak{p}\L)^{\I_1}$ est égal à $\mathfrak{p}\L^{\I_1}$,
qui est un $\zlb$-réseau de $\V^{\I_1}$.
Mais $\V^{\I_1}$ est de dimension $1$ sur $\qlb$ et le groupe $\I$ agit 
trivialement dessus.
Ainsi, on a~:
\begin{equation*}
\H^1(\I/\I_1,(\mathfrak{p}\L)^{\I_1}) = \Hom(\I/\I_1,\mathfrak{p}\L^{\I_1})
\end{equation*}
qui est trivial car le quotient $\I/\I_1$ est un groupe abélien fini et 
$\mathfrak{p}\L^{\I_1}$ un $\zlb$-module libre.
Finalement, comme on a aussi $ (\mathfrak{p}\L)^{\I}=\mathfrak{p}\L^{\I}$, 
on obtient un isomorphisme naturel~:
\begin{equation*}
\L^\I\otimes\flb \simeq \overline{\L}^\I
\end{equation*}
de $\flb$-espaces vectoriels, et même de $\Hh(1,n)$-modules. 
Le membre de gauche étant la ré\-duc\-tion mod $\ell$ du $\qlb$-caractère 
$\EuScript{L}(\widetilde{1},n)$, 
où $\widetilde{1}$ désigne le $\qlb$-caractère trivial de $\mult\F$, 
le résultat s'ensuit. 
\end{proof}

\subsection{Le cas général}
\label{CasG}

Le corps $\R$ est à nouveau quelconque, de caractéristique $\ell\neq p$.
Le théorème suivant implique et précise le théorème \ref{ASG}.

\begin{theo}
\label{PRECISION}
Soit $\s$ une représentation irréductible cuspidale de degré $m\>1$.
Soit $n\>1$ et soit $\pi\in\Irr^*_{\s,n}$. 
Il y a une unique repré\-sentation irréductible $\pi^\star$
de même support cuspidal que $\pi$ telle que~: 
\begin{equation*}
\AC(\pi)- (-1)^n \cdot \pi^\star
\end{equation*}
ne contienne pas de terme irréductible de même support cuspidal que $\pi$~;
le module simple $\FM(\pi^\star)$ est égal à $\inv(\FM(\pi))$.
\end{theo}

\begin{proof}
Pour $n\>1$ et tout $\Hh(\s,n)$-module simple $\m$, 
posons $\tm(\m)=(-1)^{n}\cdot\inv(\m)$. 
Prolongeant par linéarité, on obtient un automorphisme involutif d'algèbre 
sur $\MA_\s$, encore noté $\tm$.
Pour prouver que les isomorphismes d'algèbres $\fm\circ\ac$ et 
$\tm\circ\fm$ sont égaux, il suffit de prouver qu'ils coïncident sur un 
système générateur de $\AA_\s$.

Supposons d'abord que $\s$ est le $\flb$-caractère trivial de $\mult\F$.
D'après le lemme \ref{Lemmeunipotent}, le théorème \ref{basestd} 
et le fait que $\fm\circ\ac$ et $\tm\circ\fm$ sont des morphismes d'algèbres, 
on a~:
\begin{equation*}
\fm\circ\ac(\pi)=\tm\circ\fm(\pi)
\end{equation*}
pour tout $n\>1$ et tout $\pi\in\XA_{1,n}^*$. 
Par extension des scalaires de $\flb$ à $\R$, 
on a le même ré\-sul\-tat quand $\s$ est le 
$\R$-caractère trivial de $\mult\F$.
Par conséquent, le théorème \ref{PRECISION} est vrai lorsque $\s$ est le 
$\R$-caractère trivial de $\mult\F$, et la re\-pré\-sen\-tation 
$\Z(1,n)^\star$ est égale à $\L(1,n)$ pour tout $n\>1$.

Passons maintenant au cas général. 
Fixons une extension finie $\F'$ de $\F$ dont le cardinal 
(celui du corps résiduel dans le cas $p$-adique) 
est égal à $q_{\s}$ et notons $\s'$ le caractère trivial de 
$\F^{\prime\times}$.
On a donc $q_{\s'}=q_{\s}$, \ie qu'on est dans les conditions du paragraphe 
\ref{GreatExpectations}.
D'après le corollaire \ref{COROCHGTGP} et ce qui précède, 
la première partie du théorème \ref{PRECISION} est donc vraie.
Soit maintenant $\pi$ dans $\Irr^*_{\s,n}$ et notons $\pi'$ son image par 
$\boldsymbol{\Phi}$ (paragraphe \ref{GreatExpectations}). 
D'après la proposition \ref{ChGrFD}, on a~:
\begin{equation*}
\FM(\pi^\star) 
= (-1)^n\cdot\FM(\AC(\pi)) 
= (-1)^n\cdot\FM'(\AC'(\pi')) 
= \FM'(\pi^{\prime\star}).
\end{equation*}
Ce dernier est égal à $\inv(\FM'(\pi'))$ d'après le théorème \ref{PRECISION} 
appliqué à $\s'$, et le résultat suit du fait que $\FM'(\pi')$ est égal à 
$\FM(\pi)$. 
\end{proof}


\begin{prop}
\label{COMPLEMENT}
Soit $\s$ une représentation irréductible cuspidale,
et soit $n\>1$.
\begin{enumerate}
\item
On a $\Z(\s,n)^\star=\L(\s,n)$.
\item
Supposons que $\R$ soit égal à $\flb$ 
et qu'il y ait une $\qlb$-représentation irréductible cuspidale 
entière $\widetilde{\s}$ dont la réduc\-tion mod $\ell$ soit égale à $\s$. 
Alors $\r_\ell(\L(\widetilde{\s},n))-\L(\s,n)$ appartient à $\JA_\s$, 
\ie qu'il ne contient aucun terme irréductible de support cuspidal 
$\s+\s\nu_\s^{}+\dots+\s\nu_{\s}^{n-1}$.
\end{enumerate}
\end{prop}

\begin{proof}
Le $\Hh(\s,n)$-module $\inv(\EuScript{Z}(\s,n))$ étant égal à 
$\EuScript{L}(\s,n)$
et compte tenu de la défi\-nition de $\Z(\s,n)$ et $\L(\s,n)$, 
le théorème \ref{PRECISION} implique que $\Z(\s,n)^\star$ est égal à 
$\L(\s,n)$.
La se\-con\-de partie de la proposition suit de la conjonction du lemme 
\ref{Drl} et du théorème \ref{PRECISION}.
\end{proof}

\begin{rema}
D'après \cite[Remarque 2.10]{MSf}, toute $\flb$-représentation irréductible 
cuspidale se relève à $\qlb$ dans le cas fini, \ie que la condition de la 
proposition \ref{COMPLEMENT} est toujours remplie dans ce cas.
Dans le cas $p$-adique en revanche, il y a des $\flb$-représentations 
irréductibles cuspidales qui ne se relèvent pas à $\qlb$
(voir \cite[Exemple 3.31]{MSt}).
Cependant~:
\begin{enumerate}
\item
toute $\flb$-représentation irréductible supercuspidale se 
relève (\cite[Théorème 6.11]{MSc})~;
\item
toute $\flb$-représentation irréductible cuspidale se relève 
quand $\D=\F$ (\cite[Section 3]{MSt}).
\end{enumerate}
Dans un article ultérieur \cite{VScomptage}, 
on donnera une condition nécessaire et 
suffisante de relèvement en termes d'invariants numériques attachés 
à $\s$.
\end{rema}

La proposition suivante est un cas particulier -- le cas des algèbres de 
Hecke affines de type $\textsf{A}$ -- d'une formule de Kato 
\cite[Theorem 2]{Kato}. 

\begin{prop}
\label{ForKato}
Soit $\s$ une représentation irréductible cuspidale et soit $n\>1$.
Pour tout $\Hh(\s,n)$-module $\m$ de longueur finie, 
on a l'égalité~:
\begin{equation*}
\sum\limits_{\g} (-1)^{n-r(\g)}\cdot{\bf i}_\g\circ{\bf r}_\g(\m) = \inv(\m)
\end{equation*}
dans le groupe de Grothendieck $\MA_\s$, où $\g$ décrit les compositions de
$n$. 
\end{prop}

\begin{proof}
C'est une conséquence du lemme \ref{calculdelta} et du 
théorème \ref{PRECISION}.
\end{proof}

\begin{rema}
On reprend les notations de la section \ref{SEC3}. 
Le foncteur $\bt\circ\FM$ s'identifie à~:
\begin{equation*}
\pi\mapsto\Hom_\G(\Q^{\inv},\pi)
\end{equation*}
où $\Q^{\inv}$ est égal à $\Q$ en tant que représentation de $\G$,
et au $\Hh$-module à gauche $\Q$ tordu par $\inv$ 
en tant que $\Hh$-module à gauche.
Peut-on déterminer explicitement la structure du bimodule $\Q^{\inv}$~? 
\end{rema}

\subsection{Complément}

Nous profitons de l'appareil technique mis en place dans cet article pour 
répondre à une~ques\-tion de B.~Leclerc 
concernant les multiplicités des représentations irréductibles dans les 
re\-pré\-sen\-tations standard du théorème \ref{basestd}.

Supposons dans ce paragraphe qu'on est dans le cas $p$-adique.
Un \textit{segment formel} est une~pai\-re $[a,b]$ formée de deux entiers 
$a,b\in\ZZ$ tels que $a\<b$.
L'entier $b-a+1$ est appelé la \textit{longueur} de ce segment formel.
Si $[a,b]$ est un segment formel et si $\s$ est une représentation 
irréductible cuspidale, on pose~:
\begin{equation*}
[a,b]\boxtimes\s=[\s\nu_{\s}^{a},b-a+1]
\end{equation*}
qui est un segment au sens de la définition \ref{DEFSEG}.
Notant $\Dd$ l'ensemble des segments formels, 
on ob\-tient par linéarité une application $\upmu\mapsto\upmu\boxtimes\s$ 
de $\Dive(\Dd)$ dans $\Dive(\Dd_{\s})$.
Etant donné~:
\begin{equation}
\label{MSF}
\upmu=[a_1,b_1]+\dots+[a_r,b_r]\in\Dive(\Dd),
\end{equation}
le mutisegment $\upmu\boxtimes\s$ est dit \textit{apériodique} si 
$\R$ est de caractéristique $0$ ou si, 
pour tout $n\>1$, il y a un $k\in\ZZ$ tel que le segment 
$[\s\nu_\s^k,n]$ n'apparaît pas dans $\upmu\boxtimes\s$
(voir \cite[Définition 9.7]{MSc}).

Nous avons construit dans \cite{MSc} une application surjective $\Z$ de 
$\Dive(\Dd_{\s})$ dans $\XA_{\s}$, qui est bijec\-tive quand $\s$ est 
supercuspidale et qui 
coïncide avec la classifica\-tion de Zele\-vin\-ski 
quand $\R$ est le corps des nombres complexes.
Pour $\upmu,\upnu\in\Dive(\Dd)$, 
notons $\textsf{m}(\upmu,\upnu,\s)$ la multiplicité~de 
$\Z(\upnu\boxtimes\s)$ dans la représentation standard~:
\begin{equation}
\label{RSTD}
\Z([a_1,b_1]\boxtimes\s)\times\dots\times\Z([a_r,b_r]\boxtimes\s)
\end{equation}
où l'on écrit $\upmu$ comme en \eqref{MSF}.
D'après \cite[Théorème 9.36]{MSc}, 
on a $\Z(\upnu\boxtimes\s)\in\XA_{\s}^*$ si et seulement si 
$\upnu\boxtimes\s$ est apériodique.

Soient $\upmu$ comme en \eqref{MSF} et 
$\upnu=[c_1,d_1]+\dots+[c_s,d_s]\in\Dive(\Dd)$, 
dont les segments formels sont supposés être rangés par longueur décroissante. 
On écrit $\upmu\trianglelefteq\upnu$ 
lorsque~: 
\begin{equation*}
\sum\limits_{i\<k}(b_i-a_i+1)\<\sum\limits_{i\<k}(d_i-c_i+1)
\end{equation*}
pour tout $k\>1$. 
Ceci définit une relation d'ordre sur l'ensemble $\Dive(\Dd)$.

\begin{prop}
\label{multi10}
On a ${\sf m}(\upmu,\upmu,\s)=1$ et, 
si ${\sf m}(\upmu,\upnu,\s)\neq0$ 
alors $\upmu\trianglelefteq\upnu$.
\end{prop}

\begin{proof}
C'est une conséquence de \cite[Proposition 9.19]{MSc}.
\end{proof}

Comme au paragraphe \ref{GreatExpectations}, on fixe un corps $\F'$ 
localement compact non archimédien de carac\-téristique résiduelle $p$, 
une $\F'$-algèbre à division centrale $\D'$ et un entier $m'\>1$.

\begin{prop}
\label{Mmunu}
Soient $\s$ et $\s'$\! des représentations irréductibles cuspidales 
de $\GL_{m}(\D)$ et $\GL_{m'}(\D')$ respectivement, telles que 
$q_{\s'}=q_{\s}$ dans $\R$.
Alors~:
\begin{equation*}
{\sf m}(\upmu,\upnu,\s')={\sf m}(\upmu,\upnu,\s)
\end{equation*}
pour tous $\upmu,\upnu\in\Dive(\Dd)$ tels que 
$\upnu\boxtimes\s$ -- ou de façon équivalente $\upnu\boxtimes\s'$ -- 
soit apériodique. 
\end{prop}

\begin{proof}
On vérifie d'abord, grâce à \cite[Lemme 4.45]{MSt} et à l'égalité
$q_{\s'}=q_{\s}$ dans $\R$, 
que $\upnu\boxtimes\s$ est apériodique si et seulement si $\upnu\boxtimes\s'$ 
l'est. 
Soit $\upnu\in\Dive(\Dd)$, et supposons que $\upnu\boxtimes\s$ est apériodique. 
Alors~:
\begin{equation*}
\L_{\upnu}^{\s}=\FM(\Z(\upnu\boxtimes\s))
\end{equation*}
est un module simple. 
On pose~:
\begin{equation*}
\M_{\upmu} 
= \FM(\Z([a_1,b_1]\boxtimes\s)\times\dots\times\Z([a_r,b_r]\boxtimes\s))
= \EuScript{Z}(\s,n_1)\times\dots\times\EuScript{Z}(\s,n_r)
\end{equation*}
avec $n_i=b_i-a_i+1$ pour $i\in\{1,\dots,r\}$,
qui ne dépend que de $\upmu$ (et de $q_{\s}$).
On a donc~:
\begin{equation}
\label{TRANSMmunu}
{\sf m}(\upmu,\upnu,\s)=[\M_\upmu^{}:\L_{\upnu}^\s]
\end{equation}
où le membre de droite désigne la multiplicité de $\L_{\upnu}^\s$ dans 
$\M_\upmu^{}$.

\begin{lemm}
Le module 
$\L_{\upnu}^\s$ est l'unique module simple vérifiant les conditions 
suivantes~: 
\begin{enumerate}
\item 
Il apparaît avec multiplicité $1$ dans $\M_\upnu$.
\item
S'il apparaît dans $\M_\upmu$, alors $\upmu\trianglelefteq\upnu$.
\end{enumerate}
\end{lemm}

\begin{proof}
Appliquant $\FM$, la proposition \ref{multi10} implique 
que $\L_{\upnu}^\s$ apparaît avec multiplicité $1$ dans $\M_\upnu$. 
Ensuite, supposons que $\L_{\upnu}^\s$ apparaît dans $\M_\upmu$ avec 
multiplicité $k\>1$.
Alors $\Z(\upnu\boxtimes\s)$ apparaît dans \eqref{RSTD}
avec multiplicité $k$, \ie que ${\sf m}(\upmu,\upnu,\s)=k$.
De la proposition \ref{multi10}, on déduit que $\upmu\trianglelefteq\upnu$.

Soit maintenant $\L$ un module simple vérifiant les conditions en question. 
D'après le théorème \ref{qptf}, il existe une représentation irréductible 
$\pi\in\XA_{\s}^*$ telle que $\FM(\pi)$ soit égal à $\L$.
Ensuite, d'après \cite[Théorème 9.36]{MSc},
il existe un $\uplambda\in\Dive(\Dd)$ tel que $\Z(\uplambda\boxtimes\s)$ 
soit égal à $\pi$.
Par conséquent, on a $\L=\L_\uplambda^\s$.
Les conditions 1 et 2 impliquent que $\upnu\trianglelefteq\uplambda$.
Comme $\L$ apparaît dans $\M_\uplambda$, la condition 2 implique que 
$\uplambda\trianglelefteq\upnu$.
Par conséquent, on a $\uplambda=\upnu$ et le résultat s'ensuit.
\end{proof}

Il s'ensuit que $\L_{\upnu}^\s$ est égal à $\L_{\upnu}^{\s'}$ 
(qu'on note $\L_{\upnu}$), 
ce qui met fin à la preuve de la proposition.
\end{proof}

\begin{rema}
La proposition \ref{Mmunu} reste sans doute vraie 
sans l'hypothèse d'apériodicité, 
mais ceci nécessite d'autres méthodes car 
$\FM(\Z(\upnu\boxtimes\s))$ est nul quand 
$\upnu\boxtimes\s$ n'est pas apériodique. 
\end{rema}

\begin{rema}
Compte tenu de \cite{CG} et de \cite{AM}, 
on en déduit que les multiplicités \eqref{TRANSMmunu} 
ne dépendent que de $\upmu,\upnu$ et de l'entier $e(\s)$ 
défini par \eqref{DEFe}.
En particulier, dans le cas complexe, les multiplicités ne dépendent que de 
$\upmu$ et $\upnu$.
\end{rema}

\providecommand{\bysame}{\leavevmode ---\ }
\providecommand{\og}{``}
\providecommand{\fg}{''}
\providecommand{\smfandname}{\&}

\end{document}